\documentclass[11pt]{amsart}
\usepackage{longtable, stackrel, multicol, fancybox, stackrel, tcolorbox, mathtools}
\usepackage{bm,amsfonts,amsmath,amssymb,mathrsfs,bbm,amsthm, hyperref} 
\usepackage{graphicx, color, eurosym, eucal,palatino}
\usepackage[margin=1.25in]{geometry}
\usepackage{scrextend}
\changefontsizes{11pt}
\usepackage[normalem]{ulem}

\definecolor{darkred}{rgb}{0.6,0.2,0.2}
\definecolor{darkblue}{rgb}{0.2,0.2,0.6}
\definecolor{darkblue2}{rgb}{0.2,0.2,0.9}
\definecolor{superdarkblue}{rgb}{0.2,0.2,0.3}
\hypersetup{colorlinks=true, linkcolor=darkblue,citecolor=darkblue,
	urlcolor = superdarkblue}
\definecolor{citegreen}{rgb}{0.2,0.2,0.6}

\usepackage[english]{babel}
\usepackage{graphicx,amsmath,amssymb,color, enumerate}
\numberwithin{equation}{section}

\usepackage{color}

\newcommand\nb{\nabla}
\newcommand{\beq}{\begin{equation} \begin{split}}
\newcommand{\eeq}{\end{split} \end{equation}}

\newcommand\Omg{\Omega}
\newcommand\ext{{\rm ext}}

\makeatletter
\def\section{\@startsection{section}{1}\z@{.9\linespacing\@plus\linespacing}%
	{.7\linespacing} {\fontsize{13}{14}\selectfont\bfseries\centering}}
\def\paragraph{\@startsection{paragraph}{4}%
	\z@{0.3em}{-.5em}%
	{$\bullet$ \ \normalfont\itshape}}

\@addtoreset{equation}{section}
\makeatother

\renewcommand\and{\qquad\text{and}\qquad}

\newcommand\sm{\setminus}

\newcommand{\comm}[1]{}

\def\sfH{\mathsf{H}}

\def\bm1{\mathbbm{1}}
\def\G{\Gamma}

\def\s{\sigma}

\def\p{\partial}

\def\sfW{\mathsf{W}}

\def\Re{{\rm Re}\,}

\def\arr{\rightarrow}

\renewcommand{\gg}{{\gamma}}

\def\tt{\theta}
\def\aa{\alpha}
\def\lm{\lambda}

\def\s{\sigma}

\def\sess{\sigma_{\rm ess}}

\def\ii{{\mathsf{i}}}
\def\p{\partial}

\def\kp{\kappa}

\def\sfH{\mathsf{H}}

\def\dd{{\mathsf{d}}}

\def\sfW{\mathsf{W}}

\def\sfU{\mathsf{U}}

\newcounter{counter_a}
\newenvironment{myenum}{\begin{list}{{\rm(\roman{counter_a})}}%
{\usecounter{counter_a}
\setlength{\itemsep}{1.ex}\setlength{\topsep}{0.8ex}
\setlength{\leftmargin}{5ex}\setlength{\labelwidth}{5ex}}}{\end{list}}

\usepackage[latin1]{inputenc}
\usepackage[T1]{fontenc}

\newcommand{\eg}{{\it e.g.}\,}
\newcommand{\ie}{{\it i.e.}\,}
\newcommand{\cf}{{\it cf.}\,}

\numberwithin{figure}{section}
\numberwithin{equation}{section}
\theoremstyle{plain}
\newtheorem*{thm*}{Theorem}
\newtheorem{thm}{Theorem}[section]

\newtheorem{lem}[thm]{Lemma}
\newtheorem{prop}[thm]{Proposition}

\newtheorem{dfn}[thm]{Definition}
\theoremstyle{remark}
\newtheorem{remark}[thm]{Remark}
\theoremstyle{plain}

\newcommand{\supp}{\mathrm{supp}\,}
\newcommand{\beu}{\begin{equation*}}
\newcommand{\eeu}{\end{equation*}}
\newcommand{\besu}{\begin{equation*}
\begin{aligned}}
\newcommand{\eesu}{\end{aligned}
\end{equation*}}
\newcommand{\bes}{\begin{equation}
\begin{aligned}}
\newcommand{\ees}{\end{aligned}
\end{equation}}

\newcommand\cB{\mathcal B}

\newcommand\cH{\mathcal H}

\newcommand\cL{\mathcal L}

\newcommand\frh{\mathfrak h}

\newcommand\eps{\varepsilon}

\newcommand\ov{\overline}

\newcommand\void[1]{}

\def\ov{\overline}
\def\eps{\varepsilon}
\def\ran{{\rm ran\,}}

\def\frb{{\mathfrak b}}

      \def\dC{{\mathbb C}}

   \def\dN{{\mathbb N}}   
      \def\dR{{\mathbb R}}

   \def\dZ{{\mathbb Z}}

   \def\sfH{{\mathsf H}}

   \def\sfT{{\mathsf T}}   \def\sfU{{\mathsf U}}
   \def\sfW{{\mathsf W}}

   \def\cB{{\mathcal B}}   
      
   \def\cH{{\mathcal H}}   
      \def\cL{{\mathcal L}}

\newcommand{\dom}{\mathrm{dom}\,}

\newcommand{\dist}{\mathsf{dist}}
\renewcommand{\Re}{\mathsf{Re}}

\begin{document}

\title[An isoperimetric inequality for the perturbed Robin bi-Laplacian]{An isoperimetric inequality for the perturbed Robin bi-Laplacian in a planar exterior domain}

\author[V. Lotoreichik]{Vladimir Lotoreichik}
\address[V. Lotoreichik]{Department of Theoretical Physics, Nuclear Physics Institute, 	Czech Academy of Sciences, 25068 \v Re\v z, Czechia}
\email{lotoreichik@ujf.cas.cz}


\maketitle
\begin{abstract}
	In the present paper we introduce the perturbed two-dimensional Robin bi-Laplacian in the exterior of a bounded simply-connected $C^2$-smooth open set. The considered perturbation is of lower order and corresponds to tension. We prove that the essential spectrum of this operator coincides with the positive semi-axis and
	that the negative discrete spectrum is non-empty if, and only if,
	the boundary parameter is negative. As the main result, we obtain an isoperimetric inequality for the lowest eigenvalue of such a perturbed Robin bi-Laplacian
	with a negative boundary parameter
	in the exterior of a bounded convex planar set under the constraint on the maximum of the curvature of the boundary with the maximizer being the exterior of the disk. The isoperimetric inequality is proved under the additional assumption that to the lowest eigenvalue for the exterior of the disk corresponds a radial eigenfunction. We provide a sufficient condition in terms of the tension parameter and the radius of the disk for this property to hold.
\end{abstract}
\section{Introduction}
The spectral analysis of the bi-Laplacian on a domain originates from applications in mechanics to the study of plates. Nowadays it is an independent 
mathematical area with many challenging open problems.
The (perturbed) Robin bi-Laplacian in a bounded domain is recently introduced by
Chasman and Langford in~\cite{CL20} and
in a more general form by Buoso and Kennedy in~\cite{BK22}. A physical motivation comes from the analysis of plates with elastic response of the boundary. This operator is of fourth-order and by saying that the bi-Laplacian is  perturbed we mean that the second-order term corresponding to tension is included.   Among other results  it is shown in~\cite{BK22} that balls are critical domains for any eigenvalue under the constraints of fixed volume or of fixed area of the boundary. In~\cite{CL20} the authors proved
under fixed volume constraint with certain restrictions on the parameters involved in the definition of the operator that the second eigenvalue of the perturbed Robin bi-Laplacian is maximized by the ball. An isoperimetric inequality for the lowest eigenvalue is not yet obtained in the literature even in the case of bounded domains.

The main aim of the present paper is to prove a
global optimization result for the lowest eigenvalue of the perturbed Robin bi-Laplacian in the complementary setting of a planar exterior domain, more precisely, exterior of a bounded planar convex set. This setting can be
physically motivated by study of multiply-connected bounded plates, for which the boundary has an inner component with an elastic response and a free outer component located on a sufficiently large distance, so that the influence of the outer component of the boundary can be neglected.  In the exterior of a bounded convex set
the cut-locus is empty and the parallel coordinates are globally well defined.
We take the advantage of this property in the construction of the test function in the proof of the isoperimetric inequality. Our methods are inspired by eigenvalue optimization for the usual Robin Laplacian in exterior domains considered
by Krej\v{c}i\v{r}\'{i}k and the author in~\cite{KL18, KL20} for the lowest eigenvalue and later by Exner and the author in~\cite{EL22} for the second eigenvalue.   

In the present paper we study the perturbed Robin bi-Laplacian in the exterior $\Omg^\ext := \dR^2\sm\ov{\Omg}$ of a bounded simply-connected $C^2$-smooth domain $\Omg\subset\dR^2$. So far, we do not assume that $\Omg$ is convex. This perturbed Robin bi-Laplacian is rigorously introduced
as a self-adjoint operator in $L^2(\Omg^\ext)$ via the first representation theorem as associated to the closed, densely defined, symmetric, and lower-semibounded quadratic form in the Hilbert space $L^2(\Omg^\ext)$ given by
\begin{equation}\label{eq:form_intro}
	H^2(\Omg^\ext)\ni u\mapsto 
	\int_{\Omg^\ext}\Big(|\nb \p_1 u|^2+ |\nb \p_2 u|^2 + \tau|\nb u|^2\Big)\dd x + \gg\int_{\p\Omg}|u|^2\dd \s,
\end{equation} 
where $H^2(\Omg^\ext)$ is the second-order $L^2$-based Sobolev space on $\Omg^\ext$, $\gg\in\dR$ is the boundary parameter, $\tau\ge 0$ is the tension parameter, and $\dd\s$ stands for the surface measure on the boundary $\p\Omg$ of $\Omg$. In the following we denote this operator by $\sfH_{\tau,\gg}^{\Omg^\ext}$. In the case $\tau = 0$ we get the non-perturbed Robin bi-Laplacian on $\Omg^\ext$. The quadratic form~\eqref{eq:form_intro} is reminiscent of the one in~\cite{BK22, CL20} upon replacement of a bounded domain by an unbounded exterior domain. The eigenvalue problem for $\sfH_{\tau,\gg}^{\Omg^\ext}$ is equivalent to the following formal spectral problem (see Appendix~\ref{app:BC})
\[
	\begin{cases}
		\Delta^2 u -\tau\Delta u = \lm u,&\qquad\text{in}~\Omg^\ext,\\[0.4ex]
		\displaystyle\frac{\p^2 u}{\p\nu^2} = 0,&\qquad\text{on}~\p\Omg\\[1.2ex]
		\displaystyle\frac{\p(\Delta u)}{\p\nu} + \frac{\p}{\p\tau}\left[\frac{\p^2 u}{\p\tau\p\nu} - \kp\frac{\p u}{\p\tau}\right]-\tau\frac{\p u}{\p\nu} + \gg u = 0,&\qquad\text{on}~\p\Omg,
	\end{cases}
\]
where $\lm$ is the spectral parameter, $\frac{\p}{\p\nu}$ stands for the normal derivative on $\p\Omg$ with the normal pointing outwards of $\Omg$, $\frac{\p}{\p\tau}$ stands for the  tangential derivative on $\p\Omg$, and where $\kp$ is the curvature of $\p\Omg$ which is non-negative provided that $\Omg$ is convex.

In the general setting we prove that the essential spectrum of $\sfH_{\tau,\gg}^{\Omg^\ext}$ coincides with the interval $[0,\infty)$ and that its negative discrete spectrum is non-empty if, and only if, $\gg < 0$.
In what follows we assume that $\gg < 0$ and denote by $\lm_1^{\tau,\gg}(\Omg^\ext) < 0$ the lowest eigenvalue of $\sfH_{\tau,\gg}^{\Omg^\ext}$. Our main result concerns the optimization of this eigenvalue under the assumption that $\Omg$ is convex. 
\begin{thm}\label{thm:main}
	Let $\cB\subset\dR^2$ be the disk centred at the origin of radius $R > 0$. Assume that the parameters $\tau\ge 0$, $\gg < 0$ and $R >0$ are such that  to the lowest eigenvalue $\lm_1^{\tau,\gg}(\cB^\ext) <0$ of $\sfH_{\tau,\gg}^{\cB^\ext}$ corresponds a radial eigenfunction (this property holds, in particular, under the assumption $\tau\ge\nobreak \frac{1}{R^2}$).
	Let $\Omg\subset\dR^2$ be a bounded $C^2$-smooth convex domain whose curvature $\kp$ satisfies $\max\kp \le\nobreak \frac{1}{R}$. Then the following inequality
	\[
		\lm_1^{\tau,\gg}(\Omg^\ext) \le \lm_1^{\tau,\gg}(\cB^\ext),
	\] 
	holds, in which the equality occurs if, and only if $\Omg$ and $\cB$ are congruent.
\end{thm}
This result is proved by the min-max principle via the transplantation of a ground-state eigenfunction for the exterior of the disk onto the exterior of $\Omg$ by means of parallel coordinates. In comparison with the
proof of an isoperimetric inequality for the lowest eigenvalue for the usual Robin Laplacian in an exterior
of a bounded convex domain~\cite{KL18} a new-type term involving the curvature of the boundary appeared in the Rayleigh quotient evaluated on the test function. The presence of this term is the reason to impose a geometric condition on the curvature in the main result. Similar term involving the curvature appeared for another reason in the optimization of the second eigenvalue for the Robin Laplacian
in the exterior of a bounded convex planar set considered in~\cite{EL22} and the geometric condition in Theorem~\ref{thm:main} is reminiscent of the one in that paper.  

There are some questions related to Theorem~\ref{thm:main} left open.
It remains an open problem to verify whether the same isoperimetric inequality as in Theorem~\ref{thm:main} holds under fixed perimeter constraint
and without the convexity assumption. Another open question is to check whether for some choice of the parameters $\tau\ge0,\gg < 0, R > 0$ to the lowest eigenvalue of $\sfH_{\tau,\gg}^{\cB^\ext}$ does not correspond a radial eigenfunction. In Theorem~\ref{thm:ef_disk}\,(i) we prove
using separation of variables a partial result in this direction that to this lowest eigenvalue corresponds either a radial eigenfunction or an eigenfunction that can be represented as $f(r)e^{\pm\ii\tt}$ in the polar coordinates $(r,\tt)$ with a suitable function $f\colon(R,\infty)\arr\dR$. However, it is not clear if the second (non-radial) case ever occurs and as was already mentioned under the assumption $\tau\ge\frac{1}{R^2}$ a ground-state of $\sfH_{\tau,\gg}^{\cB^\ext}$ can be chosen to be a radial function.

In the following we review isoperimetric inequalities for the bi-Laplacian
and related results for the usual Laplacian in order to position our work in the literature. It was conjectured by Lord Rayleigh~\cite{R} that the lowest eigenvalues of the Dirichlet Laplacian  and the Dirichlet bi-Laplacian (describing the clamped plate)  are minimized by disks (or by balls in higher dimensions) among domains of fixed volume. In the case of the Laplacian this result is proved by Faber~\cite{F23} and Krahn~\cite{K24} almost a century ago. The case of the bi-Laplacian turned out to be significantly more complicated. Relying on the technique developed by Talenti in~\cite{T81}, Nadirashvili proved in~\cite{N95} the analogue of the Faber-Krahn inequality for the Dirichlet bi-Laplacian in dimension $d = 2$. Ashbaugh and Benguria~\cite{AB95} again relying on Talenti's obtained such an isoperimetric inequality in dimension $d = 3$. The same question in dimensions $d \ge 4$ remains still open. 

For the Neumann Laplacian the lowest eigenvalue is zero and the most natural question concerns optimization of the second eigenvalue. It is proved in two dimensions by Szeg\H{o}~\cite{S54} and later in all dimensions by Weinberger~\cite{W56} that the ball maximizes the second Neumann eigenvalue among domains of fixed volume. The respective isoperimetric inequality for the perturbed Neumann bi-Laplacian corresponding to a plate under tension with free boundary was proved by Chasman in~\cite{C11}.

As for the Robin boundary condition, the results are not complete even in the setting of the usual Laplacian. For the positive boundary parameter it is proved by Bossel \cite{B86} in two dimensions and generalized by Daners~\cite{D06} to higher space dimensions that the ball is the minimizer of the lowest eigenvalue of the Robin Laplacian. For the negative boundary parameter it was conjectured by Bareket~\cite{B77} that the ball is the maximizer of the lowest eigenvalue. This conjecture was disproved by Freitas and Krej\v{c}i\v{r}\'{i}k~\cite{FK15} relying on the comparison of the ball with the spherical shell in the limit of the negative boundary parameter large by absolute value. However, the disk is the maximizer of the lowest eigenvalue for the negative boundary parameter under fixer perimeter constraint in two dimensions 
according to the result by Antunes, Freitas, and Krej\v{c}i\v{r}\'{i}k~\cite{AFK17}
and the ball in higher dimensions is the maximizer under fixed area of the boundary in the class of convex domains according to the result by Bucur, Ferone, Nitsch and Trombetti~\cite{BFNT19}.
It is conjectured in~\cite{AFK17} that in two dimensions the disk is still the maximizer under fixed area in constraint in the class of simply-connected domains and that in higher dimensions the result of~\cite{BFNT19} is valid without the convexity assumption. Optimization of the second Robin eigenvalue under fixed volume constraint for the negative boundary parameter is considered by Freitas and Laugesen in~\cite{FL21, FL20}. Their technique is adapted in~\cite{CL20} to
optimization of the second eigenvalue for the
perturbed Robin bi-Laplacian with a negative boundary parameter. The results for the second eigenvalue are proved under certain restrictions on the negative boundary parameter. Optimization of the lowest eigenvalue of the (perturbed) Robin bi-Laplacian is not yet considered in the literature besides very general criticality result in~\cite{BK22}.

The structure of the paper is as follows.
In Section~\ref{sec:general} we rigorously introduce the perturbed Robin bi-Laplacian in a planar exterior domain and characterise its essential and negative discrete spectra.
Further, in Section~\ref{sec:ext_disk} we analyse the perturbed Robin bi-Laplacian in the exterior of the disk.
In Section~\ref{sec:proof} we prove the main result of the paper stated in Theorem~\ref{thm:main}. The paper is complemented by Appendix~\ref{app:BC} where the domain of the perturbed Robin bi-Laplacian is partially characterised and by Appendix~\ref{app:aux} with some auxiliary computation in the exterior of the disk.
\section{The perturbed Robin bi-Laplacian in 
a planar exterior domain}\label{sec:general}
In this section we rigorously introduce the perturbed Robin bi-Laplacian in a planar exterior domain. We characterise its essential spectrum by
constructing singular sequences and
applying a compact perturbation argument. Moreover, we show existence of negative discrete spectra for the negative boundary parameter using a test function argument. 

Recall that $\Omg\subset\dR^2$ is a bounded simply-connected $C^2$-smooth domain. The boundary of $\Omg$ is denoted by $\p\Omg$.
The complement of $\Omg$ is defined by
\[
	\Omg^\ext := \dR^2\sm\ov{\Omg}.
\]
This complement $\Omg^\ext$ is an unbounded connected domain with compact boundary $\p\Omg^\ext = \p\Omg$. The domain $\Omg^\ext$ will be referred to as the exterior domain. 

For the boundary parameter $\gg\in\dR$
and the tension parameter $\tau\ge 0$ we introduce the quadratic form in the Hilbert space $L^2(\Omg^\ext)$ by
\begin{equation}\label{eq:form}
	\frh^{\Omg^\ext}_{\tau,\gg}[u]\! :=\! \int_{\Omg^\ext}\!\!\Big[
	|\nb\p_1 u|^2 + |\nb\p_2 u|^2 + \tau|\nb u|^2\Big]\dd x + \gg\int_{\p\Omg}|u|^2\dd\s,\quad \dom \frh^{\Omg^\ext}_{\tau,\gg}\! :=\! H^2(\Omg^\ext),
\end{equation} 
where $\dd\s$ is the surface measure on the curve $\p\Omg$.
This quadratic form is reminiscent of the one introduced in~\cite[Eq. (2.1)]{BK22} in which the bounded domain is replaced by an exterior domain
and in which we set by zero the Poisson ratio and the second boundary parameter.
A similar quadratic form for a bounded domain is also considered in~\cite[Sec. 2]{CL20}.

It is clear that the form $\frh_{\tau,\gg}^{\Omg^\ext}$ is symmetric and densely defined. The latter property is a direct consequence of the density of the Sobolev space $H^2(\Omg^\ext)$ in $L^2(\Omg^\ext)$.
Since we deal with an exterior domain case, we can not directly apply~\cite[Thm. 2.1]{BK22} to show that the densely defined symmetric  quadratic from $\frh_{\tau,\gg}^{\Omg^\ext}$ is closed and semi-bounded. Below we provide a proof of these properties.
\begin{prop}\label{prop:form}
	The quadratic form $\frh_{\tau,\gg}^{\Omg^\ext}$ defined in~\eqref{eq:form} is closed and semi-bounded.
\end{prop}
\begin{proof}
	First, we analyse the case $\gg = 0$. In this case the form $\frh_{\tau,0}^{\Omg^\ext}$ is non-negative and it only remains to show that it is closed. Then we pass to the analysis of the general case.
	
	Recall that by~\cite[Thm. 3.5]{Adams} the standard norm $\|\cdot\|_{H^2(\Omg^\ext)}$ in the Sobolev space $H^2(\Omg^\ext)$ (that makes it a separable Hilbert space) is defined by
	\begin{equation}\label{eq:norm1}
		\|u\|^2_{H^2(\Omg^\ext)} := 
		\|\nb\p_{1} u\|^2_{L^2(\Omg^\ext;\dC^2)}+
		\|\nb\p_{2} u\|^2_{L^2(\Omg^\ext;\dC^2)}
		+\|\nb u\|^2_{L^2(\Omg^\ext;\dC^2)}
		+\|u\|^2_{L^2(\Omg^\ext)}.
	\end{equation} 
	Let us consider simultaneously the norm
	induced by the quadratic form $\frh_{\tau,0}^{\Omg^\ext}$ via
	\begin{equation}\label{eq:norm2}
		\|u\|^2_{\frh_{\tau,0}^{\Omg^\ext}} := \frh_{\tau,0}^{\Omg^\ext}[u] + \|u\|^2_{L^2(\Omg^\ext)}.
	\end{equation}
	For $\tau > 0$ we obviously have that the norm $\|\cdot\|_{\frh_{\tau,0}^{\Omg^\ext}}$
	is equivalent to the norm in~\eqref{eq:norm1}. The case $\tau = 0$ is slightly more subtle.
	Since the domain $\Omg$ has the uniform cone property (see~\cite[\S 4.4]{Adams}), the standard norm in the Sobolev space $H^2(\Omg^\ext)$ introduced above is also equivalent by~\cite[Cor. 4.16]{Adams} to the norm induced in~\eqref{eq:norm1}.
	Hence, in both cases $\tau >0$ and $\tau = 0$ we conclude by~\cite[Dfn. 10.2]{S12} that the quadratic form $\frh^{\Omg^\ext}_{\tau,0}$ is closed.
	
	Now we consider the case $\gg \ne 0$.
	By~\cite[Lem. 2.6]{BEL14} for $\eps > 0$ there exists a constant $C(\eps) > 0$ such that
	\begin{equation}\label{eq:BEL}
		\int_{\p\Omg}|u|^2\dd\s \le \eps\|\nb u\|^2_{L^2(\Omg^\ext;\dC^2)} + C(\eps)\|u\|^2_{L^2(\Omg^\ext)},\qquad \forall\, u\in H^2(\Omg^\ext).
	\end{equation}
	From the equivalence of the norms defined in~\eqref{eq:norm1} and in~\eqref{eq:norm2} we get that there exists a constant $c > 0$ such that
	\[
		\|\nb u\|_{L^2(\Omg^\ext;\dC^2)}^2 \le 
		c\| u\|_{\frh_{\tau,0}^{\Omg^\ext}}^2,\qquad\forall\, u\in H^2(\Omg^\ext).
	\]
	Combining the above bound with~\eqref{eq:BEL} we obtain that for any $u\in H^2(\Omg^\ext)$
	\[
		\left|\gamma\int_{\p\Omg}|u|^2\dd \s\right| \le \eps |\gamma|c\cdot\frh_{\tau,0}^{\Omg^\ext}[u] + \left(C(\eps) +\eps c|\gamma|\right)\|u\|^2_{L^2(\Omg^\ext)}.
	\]
	Choosing $\eps \in(0, \frac{1}{|\gamma|c})$ we conclude that the quadratic form
	$H^2(\Omg^\ext) \ni u\mapsto \gamma\|u|_{\p\Omg}\|^2_{L^2(\p\Omg)}$
	is bounded with respect to $\frh_{\tau,0}^{\Omg^\ext}$ with the form bound $< 1$.
	Then by \cite[Thm. VI 1.33]{Kato}
	the quadratic form $\frh_{\tau,\gamma}^{\Omg^\ext}$ is closed and bounded from below.
\end{proof}
In the following definition we introduce the perturbed Robin bi-Laplacian in $\Omg^\ext$.
\begin{dfn}\label{def:op}
	The Robin bi-Laplacian $\sfH_{\tau,\gg}^{\Omg^\ext}$ in $\Omg^\ext$ with the boundary parameter $\gamma\in\dR$
	and the tension parameter $\tau\ge0$ is defined as the unique self-adjoint operator in the Hilbert space $L^2(\Omg^\ext)$ associated with the quadratic form $\frh_{\tau,\gamma}^{\Omg^\ext}$ in~\eqref{eq:form} via the first representation theorem~\cite[Thm. VI 2.1]{Kato}.
\end{dfn}
\begin{remark}
	From the above definition it immediately follows that
	$\dom\sfH_{\tau,\gg}^{\Omg^\ext}\subset H^2(\Omg^\ext)$.
	In Appendix~\ref{app:BC} we show that for any $u\in C^\infty_0(\ov{\Omg^\ext})\cap\dom\sfH_{\tau,\gg}^{\Omg^\ext}$ holds 
	\[
		\sfH^{\Omg^\ext}_{\tau,\gg} u = \Delta^2 u-\tau\Delta u
	\]
	and we find the boundary condition
	satisfied by any $u\in C^\infty_0(\ov{\Omg^\ext})\cap\dom\sfH_{\tau,\gg}^{\Omg^\ext}$. This boundary condition can be extended to the whole domain of $\sfH_{\tau,\gg}^{\Omg^\ext}$ in the sense of distributions. We do not pursue the goal to make this analysis in the present paper, since it is not necessary for the proof of our main result. 
\end{remark}	
In the next proposition we characterise the essential spectrum of the perturbed Robin bi-Laplacian $\sfH_{\tau,\gg}^{\Omg^\ext}$.
\begin{prop}\label{prop:ess}
	Let the self-adjoint operator $\sfH_{\tau,\gg}^{\Omg^\ext}$ be as in Definition~\ref{def:op}.
	Then $\sess(\sfH_{\tau,\gg}^{\Omg^\ext}) = [0,\infty)$ for any $\gg\in\dR$
	and $\tau\ge0$.
\end{prop}
\begin{proof}
For clarity, we split the argument into two steps. First we deal with the case $\gg\ge0$ and then we address the case $\gg < 0$.
	
\noindent\emph{Step 1: the case $\gg \ge 0$.}
By Proposition~\ref{prop:BC} (in Appendix~\ref{app:BC})
the inclusion holds
$C^\infty_0(\Omg^\ext)\subset \dom\sfH_{\tau,\gg}^{\Omg^\ext}$ and for any $u\in C^\infty_0(\Omg^\ext)$ we have $\sfH^{\Omg^\ext}_{\tau,\gg} = \Delta^2 u-\tau\Delta u$. 
Let $\lm\in[0,\infty)$ be fixed. We aim at constructing a singular sequence for the operator $\sfH_{\tau,\gg}^{\Omg^\ext}$ corresponding to the point $\lm$. Let the real-valued non-negative function $\chi\in C^\infty_0(\dR^2)$ be chosen so that 
$\|\chi\|_{L^2(\dR^2)} = 1$ and $\supp\chi \subset\cB_1(0)$, where $\cB_R(x)$ denotes the disk centred at $x\in\dR^2$ of radius $R>0$.
Let the sequence of points $\{x_n\}_{n\in\dN}\subset\dR^2$ be chosen so that the sequence of disks $\{\cB_{n}(x_n)\}_{n\in\dN}$ is mutually disjoint and moreover $\cB_{n}(x_n)\subset\Omg^\ext$ for any $n\in\dN$. Let $p\in\dR^2$ be such that $\lm = |p|^4 +\tau|p|^2$.
Consider the sequence of functions
\begin{equation}
	u_n(x) = \frac{1}{n} \chi\left(\frac{x-x_n}{n}\right) e^{\ii p\cdot x}.
\end{equation}
It is easy to see that the functions $\{u_n\}_{n\in\dN}$ form an orthonormal family in the Hilbert space $L^2(\Omg^\ext)$. In particular, $u_n$ converges weakly to zero in $L^2(\Omg^\ext)$.

Now we compute the Laplacian of $u_n$
\[
	(\Delta u_n)(x) = 
	\left[
	\frac{1}{n^3}(\Delta \chi)\left(\frac{x-x_n}{n}\right) 
	+ \frac{2\ii p}{n^2}
		(\nabla\chi)\left(\frac{x-x_n}{n}\right)  -\frac{|p|^2}{n}\chi\left(\frac{x-x_n}{n}\right)\right]e^{\ii p\cdot x}. 
\]
Next we compute the bi-Laplacian of $u_n$
\[
\begin{aligned}
	\Delta^2 u_n &=
	 \left[
	 \frac{1}{n^5}(\Delta^2 \chi)\left(\frac{x-x_n}{n}\right) 
	 + \frac{2\ii p}{n^4}
	 (\nabla\Delta\chi)\left(\frac{x-x_n}{n}\right)  -\frac{|p|^2}{n^3}(\Delta\chi)\left(\frac{x-x_n}{n}\right)\right]e^{\ii p\cdot x}\\
	 &\qquad +2\ii p
	 	\left[
	 \frac{1}{n^4}(\nabla\Delta \chi)\left(\frac{x-x_n}{n}\right) 
	 + \frac{2\ii }{n^3}
	 (\nb(p\cdot\nabla\chi))\left(\frac{x-x_n}{n}\right)  -\frac{|p|^2}{n^2}(\nb\chi)\left(\frac{x-x_n}{n}\right)\right]e^{\ii p\cdot x}\\
	 &\qquad\qquad
	 +\left[
	 -\frac{|p|^2}{n^3}(\Delta \chi)\left(\frac{x-x_n}{n}\right) 
	 - \frac{2\ii |p|^2 p}{n^2}
	 (\nabla\chi)\left(\frac{x-x_n}{n}\right)  +\frac{|p|^4}{n}\chi\left(\frac{x-x_n}{n}\right)\right]e^{\ii p\cdot x}.   
\end{aligned}
\]
Combining the above expressions for
$\Delta u_n$ and $\Delta^2 u_n$ with the triangle inequality for the norm we obtain 
\[
\begin{aligned}
	\|\sfH_{\tau,\gg}^{\Omg^\ext} u_n \!-\! \lm u_n\|_{L^2(\Omg^\ext)} &\!\le\! 
	\frac{1}{n^4}\|\Delta^2\chi\|_{L^2(\dR^2)} + \frac{2|p|}{n^3}\|\nb\Delta\chi\|_{L^2(\dR^2;\dC^2)} 
	+ \frac{|p|^2}{n^2}\|\Delta \chi\|_{L^2(\dR^2)}\\
	&\quad\! +\! \frac{2|p|}{n^3}\|\nb\Delta\chi\|_{L^2(\dR^2;\dC^2)}\! +\! \frac{4|p|}{n^2}\|\nb(p\cdot\nb\chi)\|_{L^2(\dR^2;\dC^2)}\! +\! \frac{2|p|^3}{n}\|\nb\chi\|_{L^2(\dR^2;\dC^2)}\\
	&\qquad +\frac{|p|^2}{n^2}\|\Delta\chi\|_{L^2(\dR^2)} + \frac{2|p|^3}{n}\|\nb\chi\|_{L^2(\dR^2;\dC^2)}\\
	&\qquad\quad + \frac{\tau}{n^2}\|\Delta\chi\|_{L^2(\dR^2)} + \frac{2\tau|p|}{n}\|\nb\chi\|_{L^2(\dR^2;\dC^2)} \arr 0,\qquad \text{as}\,\,\,n\arr\infty.
\end{aligned}
\]
Hence, we conclude that $\{u_n\}_{n\in\dN}$ is a singular sequence for the operator $\sfH_{\tau,\gg}^{\Omg^\ext}$ corresponding to the point $\lm = |p|^4+\tau|p|^2$
and therefore by \cite[Prop. 8.11]{S12} we have $\lm\in\sess(\sfH_{\tau,\gg}^{\Omg^\ext})$. 
Since this construction can be performed for any $\lm\in[0,\infty)$ we infer that $[0,\infty)\subset\sess(\sfH_{\tau,\gg}^{\Omg^\ext})$.
From the expression for the form in~\eqref{eq:form} with $\gg \ge0 $ we easily see that the operator $\sfH_{\tau,\gg}^{\Omg^\ext}$ is non-negative. This observation results in
\[
	\s(\sfH_{\tau,\gg}^{\Omg^\ext}) =\sess(\sfH_{\tau,\gg}^{\Omg^\ext}) = [0,\infty),\qquad \forall\,\gg \ge0.
\]

\noindent \emph{Step 2: the case $\gg < 0$.}
 Let $a > 0$ be such that $-a < \inf\s(\sfH_{\tau,\gg}^{\Omg^\ext})$. Let $f,g\in L^2(\Omg^\ext)$ be arbitrary and define the functions
\[
	u:= \big(\sfH_{\tau,\gg}^{\Omg^\ext}+a\big)^{-1}f,\qquad v:= \big(\sfH_{\tau,0}^{\Omg^\ext}+a\big)^{-1}g.
\]
Clearly, we have $u,v\in H^2(\Omg^\ext)$.
Let us introduce the resolvent difference
\[
	\sfW := \big(\sfH_{\tau,\gg}^{\Omg^\ext}+a\big)^{-1} - 
	\big(\sfH_{\tau,0}^{\Omg^\ext}+a\big)^{-1}.
\]
We obtain that
\begin{equation}\label{eq:W}
\begin{aligned}
	(\sfW f,g)_{L^2(\Omg^\ext)} &= 
	\big(\big(\sfH_{\tau,\gg}^{\Omg^\ext}+a\big)^{-1}f,g\big)_{L^2(\Omg^\ext)} -
	\big(f,\big(\sfH_{\tau,0}^{\Omg^\ext}+a\big)^{-1}g\big)_{L^2(\Omg^\ext)}\\
	&
	=\big(u, \big(\sfH_{\tau,0}^{\Omg^\ext}+a\big)v)_{L^2(\Omg^\ext)}-
\big(\big(\sfH_{\tau,\gg}^{\Omg^\ext}+a\big)u, v\big)_{L^2(\Omg^\ext)}\\
&=
	\frh_{\tau,0}^{\Omg^\ext}[u,v] -\frh_{\tau,\gg}^{\Omg^\ext}[u,v] =
	-\gg(u|_{\p\Omg},v|_{\p\Omg})_{L^2(\p\Omg)}.
\end{aligned}
\end{equation}
Let us introduce the trace mapping $\G\colon H^2(\Omg^\ext)\arr L^2(\p\Omg)$, $\G u:= u|_{\p\Omg}$.  By the trace theorem~\cite[Thm. 3.37]{McL} we infer
that $\G$ is a bounded mapping and that $\ran\G = H^{3/2}(\p\Omg)$. Let us introduce the auxiliary operators $\sfT_1,\sfT_2\colon L^2(\Omg^\ext)\arr L^2(\p\Omg)$ by
\[	
\sfT_1 := \G\big(\sfH_{\tau,\gg}^{\Omg^\ext}+a\big)^{-1},\qquad\sfT_2 :=\G\big(\sfH_{\tau,0}^{\Omg^\ext}+a\big)^{-1}.
\]
In view of the inclusions $\dom\sfH_{\tau,\gg}^{\Omg^\ext}\subset H^2(\Omg^\ext)$ and $\dom\sfH_{\tau,0}^{\Omg^\ext}\subset H^2(\Omg^\ext)$ combined with the properties of the mapping $\G$ we get that $\sfT_1,\sfT_2$ are everywhere defined bounded operators.
Moreover, we infer that $\ran\sfT_j\subset H^{3/2}(\p\Omg)$, $j\in\{1,2\}$. Since $\p\Omg$ is compact
and $C^2$-smooth, it follows from~\cite[Thm. 4.2.2]{HW} that the embedding of $H^{3/2}(\p\Omg)$ into $L^2(\p\Omg)$ is compact. Hence, we conclude that the operators $\sfT_1$ and $\sfT_2$ are compact.

It follows from~\eqref{eq:W} that for any $f,g\in L^2(\Omg^\ext)$ one has
\[
	(\sfW f,g)_{L^2(\Omg^\ext)} = -\gg(\sfT_1f,\sfT_2 g)_{L^2(\p\Omg)}.  
\]
Thus, we obtain that 
\[
	\sfW = -\gg\sfT_2^*\sfT_1
\]
and therefore the operator $\sfW$ is compact. By the stability of the essential spectrum under perturbation compact in the sense of the resolvent difference we conclude that
\[
	\sess(\sfH_{\tau,\gg}^{\Omg^\ext}) = \sess(\sfH_{\tau,0}^{\Omg^\ext}) = [0,\infty).\qedhere
\]
\end{proof}	
In the next proposition we discuss the existence of the negative discrete spectrum for the operator $\sfH_{\tau,\gg}^{\Omg^\ext}$ with $\gg < 0$. In the proof of this proposition we use that the gradient in polar coordinates $(r,\tt)$ is given by
\begin{equation}\label{eq:grad}
	\nb = {\bf e}_r \p_r + {\bf e}_\tt \frac{\p_\tt}{r},
\end{equation} 
where the vectors ${\bf e}_r$ and ${\bf e}_\tt$ are defined by
\[
	{\bf e}_r 
	= 
	\begin{pmatrix} 
		\cos\tt\\ \sin\tt
	\end{pmatrix},
	\qquad
	{\bf e}_\tt 
	= 
	\begin{pmatrix} 
		-\sin\tt\\ \cos\tt
	\end{pmatrix}.
\]
In particular, it follows that the partial derivatives $\p_1$ and $\p_2$ can be expressed in polar coordinates as
\begin{equation}\label{eq:p1p2}
\begin{aligned}
	\p_1 &= 
	\cos\tt\p_r -\sin\tt\frac{\p_\tt}{r},\\
	\p_2 &= \sin\tt\p_r +\cos\tt\frac{\p_\tt}{r}.\\
\end{aligned}
\end{equation}

\begin{prop}\label{prop:ev}
	Let the operator $\sfH_{\tau,\gg}^{\Omg^\ext}$ be as in Definition~\ref{def:op}. The negative discrete spectrum of $\sfH_{\tau,\gg}^\Omg$ is non-empty for any
	$\tau\ge0$ and $\gg < 0$.
\end{prop}
\begin{proof}
	Without loss of generality we can assume that $0\in\Omg$ and 
	hence we have $0\notin\ov{\Omg^\ext}$. We denote by $\rho := \inf_{x\in\p\Omg}|x|$ the distance between the origin and $\p\Omg$.
	Consider the family of functions
	\[
		u_\aa(x) :=\exp\left(-\frac{|x|^\aa}{2}\right),\qquad \aa >0.
	\]
	The function $u_\aa$ is defined on the whole Euclidean plane, but we can also view it as a function on $\Omg^\ext$ by performing the restriction.
	It is clear that in this sense we have $u_\aa\in L^2(\Omg^\ext)$ for all $\aa > 0$. 
	
	Employing~\eqref{eq:grad} and~\eqref{eq:p1p2}
	we obtain for sufficiently small $\aa >0$ that
	\begin{equation}\label{eq:uaa_est1}
	\begin{aligned}
		I_\aa &:= \int_{\Omg^\ext}\left(|\nb\p_1 u_\aa|^2
		+|\nb \p_2 u_\aa|^2\right)\dd x\\
		& 
		\le
		\int_\rho^\infty\int_0^{2\pi}
		\left(\big|\nb\big(\cos\tt 
		\tfrac{\aa}{2} r^{\aa-1}e^{-r^\aa/2}\big)\big|^2 
		+
		\big|\nb\big(\sin\tt 
		\tfrac{\aa}{2} r^{\aa-1}e^{-r^\aa/2}\big)\big|^2\right)r\dd\tt\dd r \\
		&= 2\pi\int_{\rho}^\infty
		e^{-r^\aa}\left( \tfrac{\aa(\aa-1)}{2}r^{\aa-2}-
		\tfrac{\aa^2}{4}r^{2\aa-2} \right)^2r\dd r +2\pi\int_\rho^\infty \tfrac{\aa^2}{4} e^{-r^\aa} r^{2\aa-3}\dd r \\
		&\le 2\pi \int_{\rho}^\infty
		\left(\tfrac{\aa^2(\aa-1)^2}{4}
		r^{2\aa-3} + \tfrac{\aa^4}{16}r^{4\aa-3}
		-\tfrac{\aa^3(\aa-1)}{4}r^{3\aa-3}\right)\dd r +\tfrac{2\pi\aa^2}{4}\int_{\rho}^\infty r^{2\aa-3}\dd r\\
		&\le 2\pi\left(
		\frac{\aa^2(\aa-1)^2}{4(2-2\aa)\rho^{2-2\aa}} \!+\!\frac{\aa^4}{16(2-4\aa)\rho^{2-4\aa}} \!+\!\frac{\aa^3(1-\aa)}{4(2-3\aa)\rho^{2-3\aa}} \!+\!\frac{\aa^2}{4(2-2\aa)\rho^{2-2\aa}}
		\right).
	\end{aligned}	
	\end{equation}
	Thus, we infer that $u_\aa\in H^2(\Omg^\ext)$ for all $\aa >0$ and, moreover, $I_\aa\arr0$ as $\aa\arr0$.  
	Next, using~\eqref{eq:grad}
	we get
	\begin{equation}\label{eq:uaa_est2}
	\begin{aligned}
		K_\aa &:= \int_{\Omg^\ext}|\nb u_\aa|^2\dd x \le \frac{\pi\aa^2}{2}\int_0^\infty
		e^{-r^\aa}r^{2\aa-1}\dd r\\
		& =
		\frac{\pi\aa}{2}
		\int_0^\infty e^{-t}t\dd t
		=\frac{\pi\aa}{2}
		\arr 0,\qquad \aa \arr 0,
	\end{aligned}
	\end{equation}
	where we performed the substitution $
	t=r^\aa $. Furthermore, it follows from the Lebesgue dominated convergence theorem that
	\begin{equation}\label{eq:uaa_est3}
		L_\aa := \int_{\p\Omg} e^{-|x|^\aa}\dd\s(x)\arr e^{-1}|\p\Omg|,\qquad \aa \arr 0.
	\end{equation}
	
	Combining~\eqref{eq:uaa_est1},~\eqref{eq:uaa_est2}, and~\eqref{eq:uaa_est3} we obtain
	\begin{equation}\label{eq:hun}	
	\begin{aligned}
	\frh_{\tau,\gg}^{\Omg^\ext}[u_\aa]  
	=
	I_\aa + \tau K_\aa + \gg L_\aa \arr \gg e^{-1}|\p\Omg| < 0,\qquad\aa \arr 0.
	\end{aligned}
	\end{equation}
	It follows from~\eqref{eq:hun} that $\frh_{\tau,\gg}^{\Omg^\ext}[u_\aa] < 0$ for all sufficiently small $\aa>0$. 
	In view of the characterisation of the essential spectrum in Proposition~\ref{prop:ess}, the min-max principle~\cite[Thm. 12.1]{S12} yields the claim.  
\end{proof}

In what follows we denote by $\lm_1^{\tau,\gg}(\Omg^\ext) < 0$ the lowest eigenvalue of the operator $\sfH^{\Omg^\ext}_{\tau,\gg}$ for $\gg < 0$.

\section{The perturbed Robin bi-Laplacian in the exterior of a disk}\label{sec:ext_disk}
In this section we perform the analysis of the perturbed Robin bi-Laplacian in the exterior
$\cB^\ext := \dR^2\sm\ov{\cB}$ of the disk 
\[
	\cB = \cB_R = 
	\{x\in\dR^2\colon |x| < R\}.
\]
Our main aim is to find a sufficient condition under which a ground-state of $\sfH_{\tau,\gg}^{\cB^\ext}$ with $\gg < 0$ is given by a radial function. To this aim we perform the separation of variables in the polar coordinates and represent $\sfH_{\tau,\gg}^{\cB^\ext}$ as an orthogonal sum of one-dimensional fiber operators.  

Let us introduce the complete family
$\{\Pi_n\}_{n\in\dZ}$ of mutually orthogonal projections in the Hilbert space $L^2(\cB^\ext)$ based on the Fourier modes by
\begin{equation}\label{key}
	\Pi_n\colon L^2(\cB^\ext)\arr L^2(\cB^\ext),\quad (\Pi_n u)(r,\tt) = \frac{e^{\ii n \tt }}{2\pi}\int_0^{2\pi} u(r,\tt') e^{-\ii n\tt'}\dd\tt',\qquad n\in\dZ.
\end{equation}
The range of $\Pi_n$ can be naturally identified with the weighted $L^2$-space $L^2((R,\infty);r\dd r)$ via the unitary map
\begin{equation}\label{eq:Un}
	\sfU_n\colon\ran\Pi_n\mapsto L^2((R,\infty);r\dd r),\qquad
	(\sfU_nv)(r) :=
	\frac{1}{\sqrt{2\pi}}\int_{0}^{2\pi}
	v(r,\tt) e^{-\ii n\tt}\dd\tt.
\end{equation}
Hence, the family of projections $\{\Pi_n\}_{n\in\dZ}$ defines the orthogonal decomposition 
\begin{equation}\label{eq:ortho}
	L^2(\cB^\ext) = \bigoplus_{n\in\dZ}\ran\Pi_n\simeq
	 \bigoplus_{n\in\dZ}L^2((R,\infty);r\dd r). 
\end{equation} 
We use the abbreviation $\cH$ for the Hilbert space $\oplus_{n\in\dZ} L^2((R,\infty);r\dd r)$ and denote the respective inner product by $(\cdot,\cdot)_\cH$. We also denote by $(\cdot,\cdot)$ (respectively, by $\|\cdot\|$) the natural inner product in the Hilbert space $L^2((R,\infty);r\dd r)$ (respectively, the associated norm in the same Hilbert space); \ie
\[
	(f,g) = \int_R^\infty f(r)\ov{g(r)}r\dd r.
\]

Let the indices $n,m\in\dZ$ be fixed and let $u\in \ran\Pi_n\cap H^2(\cB^\ext)$ and $v\in\ran\Pi_m\cap H^2(\cB^\ext)$
be arbitrary. Clearly, there exist $f,g\in L^2((R,\infty);r\dd r)$ such that \[
	u(r,\tt) = \frac{e^{\ii n\tt}f(r)}{\sqrt{2\pi}}\quad\text{and}\quad 
	v(r,\tt) = \frac{e^{\ii m\tt}g(r)}{\sqrt{2\pi}}.
\]	
According to the computations in Appendix~\ref{app:aux} we get that for $n\ne m$
\begin{equation}\label{eq:ortho_form}
	\frh_{\tau,\gg}^{\cB^\ext}[u,v] = 0,\qquad
	\int_{\cB^\ext}\left(\nb\p_1 u\ov{\nb\p_1 v}+
	\nb\p_2 u\ov{\nb\p_2 v}\right)\dd x =0,\qquad\int_{\cB^\ext}\nb u\ov{\nb v}\dd x = 0.
\end{equation}
For $n\in\dZ$,  consider the following symmetric and densely defined quadratic form in the Hilbert space $L^2((R,\infty);r\dd r)$
\begin{equation}\label{eq:form_fiber}
\begin{aligned}
	\frh_{\tau,\gg,n}^{\cB^\ext}[f]& := \frh_{\tau,\gg}^{\cB^\ext}\left[\frac{f(r)e^{\ii n\tt}}{\sqrt{2\pi}}\right],\\
	\dom\frh^{\cB^\ext}_{\tau,\gg,n}& := 
		\big\{f\in L^2((R,\infty);r\dd r)\colon f(r)e^{\ii n\tt}\in H^2(\cB^\ext)\big\}.
\end{aligned}
\end{equation}
Using the computations in Appendix~\ref{app:aux}  and the expression~\eqref{eq:norm1} for the norm in the Sobolev space $H^2(\cB^\ext)$ we obtain an alternative representation for the quadratic form $\frh_{\tau,\gg,n}^{\cB^\ext}$, $n\in\dZ$,
\begin{equation}\label{eq:form_fiber2}
\begin{aligned}
	\frh_{\tau,\gg,n}^{\cB^\ext}[f]
	&\! =\!
	\int_R^\infty\left(|f''(r)|^2+\tau|f'(r)|^2\right)r \dd r\\ 
	&\quad\! +\!
	\int_R^\infty\left[
	\frac{\tau n^2|f(r)|^2}{r^2}
	\!+\!2n^2\left|\frac{f'(r)}{r}\!-\!\frac{f(r)}{r^2}\right|^2\! +\! \left|\frac{f'(r)}{r}\!-\!\frac{n^2f(r)}{r^2}\right|^2\right]r\dd r\! +\! \gg R|f(R)|^2,\\
	\dom\frh^{\cB^\ext}_{\tau,\gg,n}
	& \! = \!\big\{f\colon f,f',f''\in L^2((R,\infty);r\dd r)\big\}.
\end{aligned}
\end{equation}

\begin{lem}
	The quadratic form $\frh^{\cB^\ext}_{\tau,\gg,n}$, $n\in\dZ$, defined in~\eqref{eq:form_fiber} is closed and semi-bounded. 
\end{lem}
\begin{proof}
	Let $n\in\dZ$ be fixed.
	Recall that by Proposition~\ref{prop:form} the quadratic form $\frh_{\tau,\gg}^{\cB^\ext}$ is lower semi-bounded. Let $c\le0$ be such that $\frh_{\tau,\gg}^{\cB^\ext}[u] \ge c\|u\|^2_{L^2(\cB^\ext)}$ for all $u\in H^2(\cB^\ext)$. It follows from the definition of the form $\frh_{\tau,\gg,n}^{\cB^\ext}$ that
	$\frh_{\tau,\gg,n}^{\cB^\ext}[f] \ge 
		c\|f\|^2$ for all $f\in\dom\frh_{\tau,\gg,n}^{\cB^\ext}$.	%
	Hence, the quadratic form $\frh_{\tau,\gg,n}^{\cB^\ext}$ is also lower semi-bounded.
	
	It remains to show that the form $\frh_{\tau,\gg,n}^{\cB^\ext}$ is closed.
	Let the sequence of functions $\{f_m\}_{m\in\dN}$ in $\dom\frh_{\tau,\gg,n}^{\cB^\ext}$ be the Cauchy sequence with respect to the norm $\|\cdot\|_n$ defined by
	\[
		\|f\|^2_n := \frh_{\tau,\gg,n}^{\cB^\ext}[f]+(1-c)
		\|f\|^2.
	\]
	Hence, the sequence
	\[
		u_m(r,\tt) := \frac{e^{\ii n\tt}f_m(r)}{\sqrt{2\pi}}\in H^2(\cB^\ext),\qquad m\in\dN,
	\]
	is a Cauchy sequence with respect to the norm defined by 
	\[
		\|u\|^2_{\frh_{\tau,\gg}^{\cB^\ext}} := \frh_{\tau,\gg}^{\cB^\ext}[u] + (1-c)\|u\|^2_{L^2(\cB^\ext)}.
	\]
	Since the quadratic form $\frh_{\tau,\gg}^{\cB^\ext}$ is closed by Proposition~\ref{prop:form} we conclude that there exists a function $u\in H^2(\cB^\ext)$ such that $\|u_m-u\|_{\frh_{\tau,\gg}^{\cB^\ext}}\arr 0$ as $m\arr\infty$.
	By closedness of $\ran\Pi_n$ we infer that there exists $f\in L^2((R,\infty);r\dd r)$ such that
	\[
	u(r,\tt) = \frac{e^{\ii n\tt}f(r)}{\sqrt{2\pi}}.
	\]
	It follows from $u\in H^2(\cB^\ext)$ that $f\in \dom\frh_{\tau,\gg,n}^{\cB^\ext}$. Moreover, we immediately get that $\|f_m-f\|_n\arr 0$ as $m\arr \infty$. Thus, by~\cite[Dfn. 10.2]{S12} the quadratic form $\frh_{\tau,\gg,n}^{\cB^\ext}$ is closed.
\end{proof}
Now we are in position to define the fiber operators.
\begin{dfn}\label{def:op_fiber}
	For $n\in\dZ$, the fiber operator $\sfH_{\tau,\gg,n}^{\cB^\ext}$ acting in the Hilbert space $L^2((R,\infty);r\dd r)$ is defined as the unique self-adjoint associated via the first representation theorem with the quadratic form $\frh_{\tau,\gg,n}^{\cB^\ext}$ given in~\eqref{eq:form_fiber}.
\end{dfn}
With all the above preparations, we provide in the next proposition a decomposition of the perturbed Robin bi-Laplacian $\sfH_{\tau,\gg}^{\cB^\ext}$ acting in the exterior of the disk into an orthogonal sum of the fiber operators.
\begin{prop}\label{prop:orth}
	Let the self-adjoint perturbed Robin bi-Laplacian $\sfH_{\tau,\gg}^{\cB^\ext}$ in the Hilbert space $L^2(\cB^\ext)$ be as in Definition~\ref{def:op}. Let the self-adjoint fiber operators $\sfH_{\tau,\gg,n}^{\cB^\ext}$, $n\in\dZ$,
	in the Hilbert space $L^2((R,\infty);r\dd r)$ be as in Definition~\ref{def:op_fiber}. Let the unitary operators $\{\sfU_n\}_{n\in\dZ}$ be as in~\eqref{eq:Un}. Then the following decomposition
	\begin{equation}\label{eq:decomp}	
		\sfH_{\tau,\gg}^{\cB^\ext} = \bigoplus_{n\in\dZ}\sfU_n^{-1}\sfH_{\tau,\gg,n}^{\cB^\ext}\sfU_n
	\end{equation}
	holds with respect to the orthogonal decomposition~\eqref{eq:ortho}.
\end{prop}

\begin{proof}
Let the sequence of functions $f_n\in\dom\sfH_{\tau,\gg,n}^{\cB^\ext}$, $n\in\dZ$, be such that the following condition
\begin{equation}\label{eq:cond_fn}
	\sum_{n\in\dZ} \big(\|\sfH_{\tau,\gg,n}^{\cB^\ext}f_n\|^2 + \|f_n\|^2\big) < \infty
\end{equation}
holds. This is equivalent to the fact that $\oplus_{n\in\dZ}f_n$ is a generic element of $\dom(\oplus_{n\in\dZ}\sfH_{\tau,\gg,n}^{\cB^\ext})$.

Consider the function $u = \sum_{n\in\dZ}\sfU_n^{-1}f_n$,
which in view of condition~\eqref{eq:cond_fn} is well defined and belongs to $L^2(\cB^\ext)$.
 Since $\dom\sfH_{\tau,\gg,n}^{\cB^\ext}\subset \dom\frh_{\tau,\gg,n}^{\cB^\ext}$ it follows from the equivalence of the characterisations for the domains of $\frh_{\tau,\gg,n}^{\cB^\ext}$ in~\eqref{eq:form_fiber} and in~\eqref{eq:form_fiber2} that $\sfU_n^{-1}f_n\in H^2(\cB^\ext)$ for all $n\in\dZ$. 

Let $c\le 0$ be such that
$\frh_{\tau,\gg}^{\cB^\ext}[u] \ge c\|u\|^2_{L^2(\cB^\ext)}$ for all $u\in H^2(\cB^\ext)$. Such a constant $c\le0$ exists thanks to lower-semiboundedness of the quadratic form $\frh_{\tau,\gg}^{\cB^\ext}$ shown in Proposition~\ref{prop:form}. Since the quadratic form $\frh_{\tau,\gg}^{\cB^\ext}$ is also closed by Proposition~\ref{prop:form} there exists a constant $A >0$ such that
\begin{equation}\label{eq:estimate_norm}
	\|u\|^2_{H^2(\cB^\ext)} \le A\big(
	\frh_{\tau,\gg}^{\cB^\ext}[u]+ (1-c)\|u\|^2_{L^2(\cB^\ext)}\big),
\end{equation}
where we use the expression~\eqref{eq:norm1} for the norm in the Sobolev space $H^2(\cB^\ext)$.
Further, we obtain 
\[	
\begin{aligned}
	\|u\|^2_{H^2(\cB^\ext)} & = \sum_{n\in\dZ} \|\sfU_n^{-1}f_n\|^2_{H^2(\cB^\ext)} \le A\sum_{n\in\dZ} \big(\frh_{\tau,\gg}^{\cB^\ext}[\sfU_n^{-1}f_n]+(1-c)\|\sfU_n^{-1}f_n\|^2_{L^2(\cB^\ext)}\big)\\
	&=
	A\sum_{n\in\dZ} \Big((\sfH_{\tau,\gg,n}^{\cB^\ext}f_n,f_n)+(1-c)\|f_n\|^2\Big)\\
	&\le
	A\sum_{n\in\dZ} \Big(2\|\sfH_{\tau,\gg,n}^{\cB^\ext}f_n\|^2+
	(3-c)\|f_n\|^2\Big) <\infty,
\end{aligned} 
\]
where we used the orthogonality properties~\eqref{eq:ortho_form}
and the expression~\eqref{eq:norm1} for the norm in the Sobolev space $H^2(\cB^\ext)$ in the first step, the inequality~\eqref{eq:estimate_norm} in the second step,
the definition~\eqref{eq:form_fiber} and the first representation theorem in the third step, the Cauchy-Schwarz inequality in the fourth step, and the condition~\eqref{eq:cond_fn} in the last step. From the last estimate we conclude that $u\in H^2(\cB^\ext)$.

By~\cite[p. 77]{McL} the space $C^\infty_0(\ov{\cB^\ext})$ is dense in $H^2(\cB^\ext)$ and thus is a core for the quadratic form $\frh_{\tau,\gg}^{\cB^\ext}$.
Let $v\in C^\infty_0(\ov{\cB^\ext})$ be arbitrary. Hence, we find that $v = \sum_{n\in\dZ}\sfU_n^{-1}g_n$ for $g_n := \sfU_n\Pi_n v\in L^2((R,\infty);r\dd r)$. It also easily follows that $\sfU_n^{-1}g_n\in C^\infty_0(\ov{\cB^\ext})\subset H^2(\cB^\ext)$ for all $n\in\dZ$. In view of the equivalence between~\eqref{eq:form_fiber} and~\eqref{eq:form_fiber2} we get that $g_n\in \dom\frh_{\tau,\gg,n}^{\cB^\ext}$ for all $n\in\dZ$.
Finally, we obtain  with $u = \sum_{n\in\dZ}\sfU_n^{-1}f_n$ $v=\sum_{n\in\dZ}\sfU_n^{-1}g_n$ defined as above that
\[
	\begin{aligned}
	\frh_{\tau,\gg}^{\cB^\ext}[u,v] 
	&= 
	\sum_{n\in\dZ}\frh_{\tau,\gg,n}^{\cB^\ext}[f_n,g_n]\\
	& = \left(\left(\oplus_{n\in\dZ}\sfH_{\tau,\gg,n}^{\cB^\ext}\right)(\oplus_{n\in\dZ}f_n),\oplus_{n\in\dZ}g_n\right)_{\cH}
	= \left(\left(\oplus_{n\in\dZ}
	\sfU_n^{-1}\sfH_{\tau,\gg,n}^{\cB^\ext}\sfU_n\right) u,v\right)_{L^2(\cB^\ext)},
	\end{aligned}
\]
where we used the first orthogonality property in~\eqref{eq:ortho_form} and the definition of the form $\frh_{\tau,\gg,n}^{\cB^\ext}$ in the first step, the first representation theorem in the second step, and the definition of the unitary map $\sfU_n$ in the last step.
Hence, we conclude again using the first representation theorem that $u\in\dom\sfH_{\tau,\gg}^{\cB^\ext}$ and that $\sfH_{\tau,\gg}^{\cB^\ext}u = \big(\oplus_{n\in\dZ}\sfU_n^{-1}\sfH_{\tau,\gg,n}^{\cB^\ext}\sfU_n\big)u$. Since $u$ is a generic element in the domain of $\oplus_{n\in\dZ}\sfU_n^{-1}\sfH_{\tau,\gg,n}^{\cB^\ext}\sfU_n$  
we infer that
$\sfH_{\tau,\gg}^{\cB^\ext}$ is an extension of the operator $\oplus_{n\in\dZ}\sfU_n^{-1}\sfH_{\tau,\gg,n}^{\cB^\ext}\sfU_n$. However, since both
these operators are self-adjoint, they coincide and we have
\[	
	\sfH_{\tau,\gg}^{\cB^\ext} = \bigoplus_{n\in\dZ}\sfU_n^{-1}\sfH_{\tau,\gg,n}^{\cB^\ext}\sfU_n,
\]
by which the proof is complete.
\end{proof}
Recall that by Proposition~\ref{prop:ess} we have $\sess(\sfH_{\tau,\gg}^{\cB^\ext}) =[0,\infty)$ and by Proposition~\ref{prop:ev} the lowest spectral point of $\sfH_{\tau,\gg}^{\cB^\ext}$ for $\gg < 0$ is a negative discrete eigenvalue.
In the last theorem of this section we identify to which fiber corresponds the lowest eigenvalue of the perturbed Robin bi-Laplacian
$\sfH_{\tau,\gg}^{\cB^\ext}$ for $\gg < 0$. Our characterisation is only partial and depends on the value of $\tau$
and the radius of the disk. It is worth to mention that the fiber operators labelled by $n$ and $-n$ are  equal for any $n\in\dZ$.
\begin{thm}\label{thm:ef_disk}
	Let the operator $\sfH^{\cB^\ext}_{\tau,\gg}$ with $\gg < 0$ be as in Definition~\ref{def:op}. Then the following hold. 
	\begin{myenum}
		\item In the orthogonal decomposition~\eqref{eq:decomp} the lowest eigenvalue $\lm_1^{\tau,\gg}(\cB^\ext) < 0$
		of $\sfH^{\cB^\ext}_{\tau,\gg}$ does not correspond to the fibers labelled by $n\notin \{-1,0,1\}$.  
		\item Assume, in addition, that $\tau \ge \frac{1}{R^2}$. Then the lowest eigenvalue $\lm_1^{\tau,\gg}(\cB^\ext) < 0$ of $\sfH_{\tau,\gg}^{\cB^\ext}$ corresponds only to the fiber labelled by $ n= 0$. 
	\end{myenum}
\end{thm}
\begin{remark}
	In general, three cases are possible: 
	\begin{myenum}
	\item [{\rm (i)}] the lowest eigenvalue corresponds  only to the fiber labelled by $n = 0$;
	\item [{\rm (ii)}] the lowest eigenvalue corresponds both to the fibers labelled by $n = 0$ and $n =\pm1$;
	\item [{\rm(iii)}] the lowest eigenvalue corresponds only to the fibers labelled by $n =\pm 1$.
	\end{myenum}
	In the case (i) any eigenfunction corresponding to the lowest eigenvalue is radial. In the case (ii)
	there is a radial eigenfunction
	corresponding to the lowest eigenvalue
	and an eigenfunction admitting the representation $f(r)e^{\pm \ii\tt}\in H^2(\cB^\ext)$
	in polar coordinates with a
	non-trivial function $f\colon (R,\infty)\arr\dR$. 
	In the case (iii)
	any eigenfunction
	corresponding to the lowest eigenvalue
	admits the representation $f(r)e^{\pm \ii\tt}\in H^2(\cB^\ext)$
	in polar coordinates with a
	non-trivial function $f\colon (R,\infty)\arr\dR$. 
	It remains an open problem whether
	the cases (ii) and (iii) can
	occur for some values of the parameters $\tau,\gg$ and $R$.
\end{remark}
\begin{proof}[Proof of Theorem~\ref{thm:ef_disk}]
	Let $f\in\dom\frh_{\tau,\gg,n}^{\cB^\ext}$ be arbitrary. In particular, one has $f(r)e^{\ii n\tt}\in H^2(\cB^\ext)$. In view of the Sobolev embedding~\cite[Thm. 5.4, Case C]{Adams} the function $f$ is bounded. Using the representation~\eqref{eq:form_fiber2} and performing the integration by parts we obtain
	\[
	\begin{aligned}
		\frh_{\tau,\gg,n}^{\cB^\ext}[f] &= 
		\int_R^\infty\big(|f''(r)|^2+
		\tau|f'(r)|^2\big)r\dd r\\
		&\quad+
		\int_R^\infty
		\left[\frac{\tau n^2|f(r)|^2}{r^2}+2n^2
		\left|\frac{f'(r)}{r}-\frac{f(r)}{r^2}\right|^2 +\left|\frac{f'(r)}{r}-\frac{n^2f(r)}{r^2}\right|^2\right]r\dd r+\gg R|f(R)|^2\\
		&=\int_R^\infty\big(|f''(r)|^2+
		\tau|f'(r)|^2\big)r\dd r\\
		&\qquad+
		\int_R^\infty
		\left[\frac{\tau n^2|f(r)|^2}{r^2}+2n^2
		\left|\frac{f'(r)}{r}-\frac{f(r)}{r^2}\right|^2  +\frac{|f'(r)|^2}{r^2}+
		\frac{n^4|f(r)|^2}{r^4}\right]r\dd r\\
		&\qquad\qquad-n^2\int_R^\infty \frac{2\Re(f'(r)\ov{f(r)})}{r^2} \dd r +\gg R|f(R)|^2\\
		&
		=
		\int_R^\infty\big(|f''(r)|^2+
		\tau|f'(r)|^2\big)r\dd r\\
		&\qquad
		+\int_R^\infty\left[
		2n^2
		\left|\frac{f'(r)}{r}-\frac{f(r)}{r^2}\right|^2 +\frac{|f'(r)|^2}{r^2}+
		\left(\frac{\tau n^2}{r^2}+\frac{n^4-2n^2}{r^4}\right)|f(r)|^2\right)r\dd r\\
		&\qquad\qquad
		+\left(\frac{n^2}{R^2}+\gg R\right)|f(R)|^2,
	\end{aligned}
	\]
	where we used that $2\Re(f'(r)\ov{f(r)}) = (|f(r)|^2)'$.
	From the above computation and the inequality $n^4-2n^2> 0$ for $|n|\ge 2$ it follows that for any $n\notin\{-1,0,1\}$ one has
	\[
		\frh_{\tau,\gg,n}^{\cB^\ext}[f] > \frh_{\tau,\gg,0}^{\cB^\ext}[f],\qquad \forall\, f\in\dom\frh_{\tau,\gg,0}^{\cB^\ext} = \dom\frh_{\tau,\gg,n}^{\cB^\ext}, f\ne 0.
	\]
	Hence, we conclude from the min-max principle that the lowest eigenvalue of $\sfH_{\tau,\gg}^{\cB^\ext}$ can not correspond to the fiber labelled by $n\notin \{-1,0,1\}$. Thus, we have shown the claim of~(i).
	
	Under the assumption $\tau \ge\frac{1}{R^2}$ the function
	\[
		(R,+\infty)\ni r\mapsto \frac{\tau}{r^2} -\frac{1}{r^4}
	\]
	is positive. Hence, it again follows from the above computation for $\frh_{\tau,\gg,n}^{\cB^\ext}[f]$ that
	\[
		\frh_{\tau,\gg,\pm1}^{\cB^\ext}[f] > \frh_{\tau,\gg,0}^{\cB^\ext}[f],\qquad \forall f\in\dom\frh_{\tau,\gg,0}^{\cB^\ext}=\frh_{\tau,\gg,\pm1}^{\cB^\ext}, f\ne0.
	\]
	Applying again the min-max principle we conclude that under the assumption $\tau\ge\frac{1}{R^2}$ the lowest eigenvalue of $\sfH_{\tau,\gg}^{\cB^\ext}$ corresponds only to the fiber labelled by $n = 0$. Thus, we have shown the claim of (ii).
\end{proof}

\section{Proof of Theorem~\ref{thm:main}}
\label{sec:proof}
\noindent\emph{Step 1: introducing the parallel coordinates.}
Without loss of generality we assume that $\Omg$ is not congruent to the disk $\cB$.
We parametrize the boundary of the convex domain
$\Omg\subset\dR^2$ by the unit-speed mapping
$\s\colon[0,L]\arr\dR^2$ (\ie $|\s'| \equiv 1$)
in the clockwise direction,
where $L$ is the length of $\p\Omg$. The unit tangential vector to the boundary is defined by $\tau(s):=\s'(s) = (\tau_1(s),\tau_2(s))^\top$.
The outer unit normal vector is given by 
$\nu(s) = (\nu_1(s),\nu_2(s))^\top:= (-\tau_2(s),\tau_1(s))^\top$. Thanks to convexity of $\Omg$ the mapping
\[
	\cL\colon [0,L]\times\dR_+\arr\Omg^\ext,\qquad
	\cL(s,t) := \s(s)+t\nu(s),
\]
is a bijection and it defines parallel coordinates
$(s,t)$ on $\Omg^\ext$; \cf~\cite[Sec. 4]{KL18}. Let us recall the Frenet formulas (see \eg~\cite[Sec. 1.4]{Kl})
\begin{equation}\label{eq:Frenet}
	\tau'(s) = -\kp(s)\nu(s),\qquad \nu'(s) = \kp(s)\tau(s),
\end{equation}
where $\kp\colon[0,L]\arr[0,\infty)$ 
is the curvature of $\p\Omg$. Note also
that $C^2$-smoothness of $\p\Omg$ yields
that the curvature of $\p\Omg$ is a continuous function on $[0,L]$ and $\kp(0) = \kp(L)$. Using the Frenet formulas we find that the Jacobian $J_\cL$ of the mapping $\cL$ is given by
\[
	J_\cL(s,t) = \begin{pmatrix} \tau_1(s)\big(1+t\kp(s)\big) & -\tau_2(s)\\
	\tau_2(s)\big(1+t\kp(s)\big) & \tau_1(s)
	\end{pmatrix}.
\]
Hence, the Jacobian determinant can be computed as
\[
	\det J_\cL(s,t) = \big(\tau_1^2(s) + \tau_2^2(s)\big)\big(1+t\kp(s)\big) = 1+t\kp(s).
\]
It remains to express the partial derivatives $\p_1$ and $\p_2$ in terms of $\p_s$ and $\p_t$. This transform is standard and we provide it only for convenience of the reader.
Let $x  = \cL(s,t)$. Using the chain rule for the differentiation and the Frenet formulas we obtain for any $u \in H^2(\Omg^\ext)$
\[
\begin{aligned}
	\big(\p_s[ u\circ\cL]\big)(s,t) & = \p_1 u(x)\tau_1(s)\big(1+\kp(s)t\big) + \p_2 u(x)\tau_2(s)\big(1+\kp(s)t\big),\\
	\big(\p_t[u\circ\cL]\big)(s,t) & = \p_1 u(x)\nu_1(s) + \p_2 u(x)\nu_2(s).  
\end{aligned}
\]
The above system can be rewritten as
\[
\begin{aligned}
\frac{\big(\p_s[ u\circ\cL]\big)(s,t)}{1+\kp(s)t} & = \p_1 u(x)\tau_1(s) + \p_2 u(x)\tau_2(s),\\
\big(\p_t[u\circ\cL]\big)(s,t) & = \p_1 u(x)\nu_1(s) + \p_2 u(x)\nu_2(s).  
\end{aligned}
\]
Solving this linear system and simplifying the notation ($\p_s u = \p_s (u\circ\cL)$, $\p_t u = \p_t(u\circ\cL)$) we find
\[
\begin{aligned}
	\p_j u &= \tau_j(s)\frac{\p_s u}{1+\kp(s)t}+\nu_j(s)\p_tu,\qquad j=1,2.
\end{aligned}	
\]
In particular, the gradient can be expressed in $(s,t)$-coordinates as
\begin{equation}\label{eq:gradient_st}
	\nb = \frac{\tau(s)}{1+\kp(s)t}\p_s + \nu(s)\p_t.
\end{equation}
\noindent\emph{Step 2: test function argument.}
Recall that we assumed that the radius $R > 0$, the boundary parameter $\gg <0$ and the tension parameter $\tau\ge0$ are such that to the lowest eigenvalue of $\sfH^{\cB^\ext}_{\tau,\gg}$ corresponds a radial (real-valued) eigenfunction
$u_\circ\in H^2(\cB^\ext)$. Recall also that by Theorem~\ref{thm:ef_disk}\,(ii) this property holds, in particular, under the assumption $\tau\ge\frac{1}{R^2}$. In view of the decomposition in Proposition~\ref{prop:orth} there exists a non-trivial function $f_\circ\colon(0,\infty)\arr\dR$ such that $f_\circ,f_\circ',f_\circ''\in L^2((0,\infty);(r+R)\dd r)$ and that $u_\circ(r,\tt) = f_\circ(r-R)$. It follows from the min-max principle that
\begin{equation}\label{eq:char_ev_disk}
	\lm_1^{\tau,\gg}(\cB^\ext) = 
	\frac{\displaystyle\int_0^\infty\left(|f''_\circ(t)|^2
		+\tau|f'_\circ(t)|^2
		+\frac{|f'_\circ(t)|^2}{(t+R)^2}  \right)(t+R)\dd t +\gg R|f_\circ(0)|^2}{\displaystyle\int_0^\infty
		|f_\circ(t)|^2(t+R)\dd t}.
\end{equation}
Let $L_\circ$ be the length of the circle $\p\cB$.
Using the total curvature identity
$\int_0^L\kp(s)\dd s = 2\pi$
combined with the inequality $\kp \le\frac{1}{R}$ (strict on a set of positive measure) we find that 
\[
	L_\circ = 2\pi R = R\int_0^L \kp(s)\dd s < L.
\]

Let us introduce the test function on $\Omg^\ext$ (in the coordinates $(s,t)$ defined in Step 1) by the formula 
\[
	u_\star(s,t) := f_\circ(t).
\] 
Using the expression for the gradient in~\eqref{eq:gradient_st}, applying the Frenet formulas and employing the total curvature identity we get
\begin{equation}\label{eq:ustar1}
\begin{aligned}
	&\int_{\Omg^\ext}\left(
	|\nb\p_1 u_\star|^2+	|\nb\p_2 u_\star|^2
	\right)\dd x=\\
	&\qquad  =
	\int_0^\infty\int_0^L
	\left(
	|\nb(\nu_1(s)f_\circ'(t))|^2 
	+
	|\nb(\nu_2(s)f_\circ'(t))|^2 
	\right)(1+\kp(s)t)\dd s\dd t \\
	&\qquad=
	\int_0^\infty\int_0^L
	\left(|f_\circ''(t)|^2 + \frac{\kp^2(s)|f_\circ'(t)|^2}{(1+\kp(s)t)^2}\right)(1+\kp(s)t)\dd s\dd t\\
	&\qquad =
	\int_0^\infty|f_\circ''(t)|^2(L+2\pi t)\dd t + \int_0^\infty\int_0^L \frac{\kp^2(s)}{1+\kp(s)t}|f_\circ'(t)|^2\dd s \dd t\\
	&\qquad < 
	\int_0^\infty|f_\circ''(t)|^2(L+2\pi t)\dd t + L\int_0^\infty \frac{|f_\circ'(t)|^2}{R(R+t)}\dd t <\infty,
\end{aligned}\end{equation}
where we used in the penultimate step $\kp(s) \le \frac{1}{R}$ (the inequality being strict on a set of positive measure) and that the function $x\mapsto\frac{x^2}{1+tx}$ is strictly increasing on $(0,\infty)$.
Analogously we find that
\begin{equation}\label{eq:ustar2}
\begin{aligned}
	\int_{\Omg^\ext} |\nb u_\star|^2\dd x &= 
	\int_0^\infty\int_0^L |f_\circ'(t)|^2(1+\kp(s)t)\dd s\dd t=
	\int_0^\infty |f_\circ'(t)|^2(L+2\pi t)\dd t < \infty,\\
	\int_{\Omg^\ext} |u_\star|^2\dd x& = 
	\int_0^\infty\int_0^L |f_\circ(t)|^2(1+\kp(s)t)\dd s\dd t =
	\int_0^\infty |f_\circ(t)|^2(L+2\pi t)\dd t < \infty.
\end{aligned}
\end{equation}
As a consequence of~\eqref{eq:ustar1} and~\eqref{eq:ustar2} we get that $u_\star\in H^2(\Omg^\ext) = \dom\frh_{\tau,\gg}^{\Omg^\ext}$ and applying the min-max principle we arrive at the bound
\[
\begin{aligned}
	\lm_1^{\tau,\gg}(\Omg^\ext)
	& \le \frac{\frh_{\tau,\gg}^{\Omg^\ext}[u_\star]}{\|u_\star\|^2_{L^2(\Omg^\ext)}}\\
	&< \frac{\displaystyle\int_0^\infty
		\left[\left(|f_\circ''(t)|^2+\tau|f'_\circ(t)|^2\right)\left(R+\frac{2\pi R t}{L}\right)+
		\frac{|f_\circ'(t)|^2}{R+t}\right]\dd t+\gg R |f_\circ(0)|^2 }{\displaystyle\int_0^\infty |f_\circ(t)|^2\left(R +\frac{2\pi Rt}{L}\right)\dd t}\\
	&< 
	\frac{\displaystyle\int_0^\infty
		\left[\left(|f_\circ''(t)|^2+\tau|f'_\circ(t)|^2\right)\left(R+\frac{2\pi R t}{L_\circ}\right)+
		\frac{|f_\circ'(t)|^2}{R+t}\right]\dd t+\gg R |f_\circ(0)|^2 }{\displaystyle\int_0^\infty |f_\circ(t)|^2\left(R +\frac{2\pi Rt}{L_\circ}\right)\dd t} \\
	& =
	\frac{\displaystyle\int_0^\infty\left(|f''_\circ(t)|^2
		+\tau|f'_\circ(t)|^2
		+\frac{|f'_\circ(t)|^2}{(t+R)^2}  \right)(t+R)\dd t +\gg R|f_\circ(0)|^2}{\displaystyle\int_0^\infty
		|f_\circ(t)|^2(t+R)\dd t} = \lm_1^{\tau,\gg}(\cB^\ext),
\end{aligned}
\]
where we used in between that $L_\circ < L$ and that $\lm_1^{\tau,\gg}(\cB^\ext) < 0$ and employed the characterisation~\eqref{eq:char_ev_disk} in the last step.
\section*{Acknowledgements}
The author was supported by the grant No.~21-07129S 
of the Czech Science Foundation.
\begin{appendix}
\section{Boundary conditions for the domain of $\sfH_{\tau,\gg}^{\Omg^\ext}$}\label{app:BC}
In this appendix we provide a derivation of the boundary condition for the operator domain of $\sfH_{\tau,\gg}^{\Omg^\ext}$. We restrict the analysis to functions in $\dom\sfH_{\tau,\gg}^{\Omg^\ext}$ which are simultaneously smooth up to the boundary. This boundary condition is reminiscent of the one obtained for a bounded domain in~\cite{C11, CL20}. We provide here details for convenience of the reader. It should be emphasized that in this appendix we do not assume that $\Omg$ is convex.

Let $\s\colon[0,L]\arr\dR^2$ be the arc-length clockwise parametrization of $\p\Omg$ (\ie $|\dot\s| \equiv 1$). The unit normal vector at the point $\s(s)$ pointing outwards of $\Omg$ is denoted by $\nu(s)$
and the respective tangential vector is denoted by $\tau(s)$ as in Step 1 of the proof of Theorem~\ref{thm:main}. We also denote by $\kp$ the curvature of $\p\Omg$ with the convention that the curvature is non-negative if $\Omg$ is convex. Note that the curvature is sign-changing for non-convex $\Omg$.
The Frenet formulas~\eqref{eq:Frenet} are clearly valid without the convexity assumption.

Recall that by ~\cite[Thm. 5.25]{Lee} (see also~\cite[App. B]{BEHL17}) the mapping 
\begin{equation}\label{eq:mapping2}
	[0,L]\times (0,\eps)\ni(s,t)\mapsto \s(s)+t\nu(s)
\end{equation}
is injective for $\eps \in (0, (\|\kp\|_{L^\infty})^{-1})$ sufficiently small. Thus, it defines coordinates $(s,t)$ in the neighbourhood of the boundary. With a slight abuse of notation for a function $u\colon\Omg^\ext\arr\dC$ we use the abbreviation $u(s,t) = u(\s(s)+t\nu(s))$ for $s\in[0,L]$ and $t\in(0,
\eps)$. Hence, in the neighbourhood $
\Omg_\eps^\ext := \{x\in\Omg^\ext\colon \dist(x,\p\Omg) < \eps\}$  of $\p\Omg$ the partial derivatives $\p_s$ and $\p_t$ are well defined. Mimicking the computations in Step 1 of the proof of Theorem~\ref{thm:main} we get the same expression for the gradient
\begin{equation}\label{eq:gradst}
	\nb = \frac{\tau(s)}{1+\kp(s) t}\p_s + \nu(s)\p_t,
\end{equation}
which is valid on $\Omg^\ext_\eps$.
\begin{prop}\label{prop:BC}
	Let the operator $\sfH_{\tau,\gg}^{\Omg^\ext}$ be as in Definition~\ref{def:op}. Then
	\[
	\begin{aligned}
		&C^\infty_0(\ov{\Omg^\ext})\cap
		\dom\sfH_{\tau,\gg}^{\Omg^\ext} = 
		\{u\in C^\infty_0(\ov{\Omg^\ext})\colon
		(\p_{tt} u)(s,0) = 0,\\ 
		&\qquad\qquad\qquad
		\Big[\p_t(\Delta u) +
		\p_s\big[\p_{st}u-\kp\p_s u\big] - \tau\p_t u+ \gamma u\Big](s,0) = 0,\, s\in[0,L]\Big\},  
	\end{aligned}
	\]
	where $\Delta u$ is computed in the Cartesian coordinates $(x_1,x_2)$.
	Moreover, for any $u\in C^\infty_0(\ov{\Omg^\ext})\cap\dom\sfH_{\tau,\gg}^{\Omg^\ext}$ holds $\sfH_{\tau,\gg}^{\Omg^\ext} = \Delta^2 u - \tau\Delta u$.
\end{prop}
\begin{proof}
	Let $u\in C^\infty_0(\ov{\Omg^\ext})\cap\dom\sfH_{\tau,\gg}^{\Omg^\ext}$. For any $\phi \in C^\infty_0(\Omg^\ext)$ we find via integration by parts 
	\[
		\frh_{\tau,\gg}^{\Omg^\ext}[u,\phi] = (\Delta^2 u-\tau\Delta u,\phi)_{L^2(\Omg^\ext)}
			 = (\sfH_{\tau,\gg}^{\Omg^\ext} u,\phi)_{L^2(\Omg^\ext)}.
	\]
	Hence, we conclude from density of $C^\infty_0(\Omg^\ext)$ in $L^2(\Omg^\ext)$ that $\sfH_{\tau,\gg}^{\Omg^\ext} u = \Delta^2 u-\tau\Delta u$.
	
	Let $u\in C^\infty_0(\ov{\Omg^\ext})$ and $v\in H^2(\Omg^\ext)$ be arbitrary. Taking into account that the support of $u$ is compact we
	can apply~\cite[Thm. 1.5.3.1]{Gr} 
	to integrate by parts
	\begin{equation}\label{eq:huv}
	\begin{aligned}
		\frh_{\tau,\gg}^{\Omg^\ext}[u,v] &=
		(\nb\p_1 u,\nb\p_1 v)_{L^2(\Omg^\ext;\dC^2)} 
		+		(\nb\p_2 u,\nb\p_2 v)_{L^2(\Omg^\ext;\dC^2)}\\
		&\qquad\qquad\qquad\qquad\qquad + \tau(\nb u,\nb v)_{L^2(\Omg^\ext;\dC^2)}+ \gg\int_{\p\Omg} u\ov{v}\dd \s \\
		&=
		(-\nb\Delta u,\nb v)_{L^2(\Omg^\ext;\dC^2)}
		-\tau(\Delta u, v)_{L^2(\Omg^\ext)}\\
		&\qquad\! -\!
		\int_{\p\Omg} \p_\nu(\p_1 u)\ov{\p_1 v}\dd\s
		-\!
		\int_{\p\Omg} \p_\nu(\p_2 u)\ov{\p_2 v}\dd\s
		-\!\tau\int_{\p\Omg} \p_\nu u\ov{v}\dd\s
		\!+\gamma\int_{\p\Omg}u\ov{v}\dd\s\\
		&=
		(\Delta^2 u-\tau\Delta u,v)_{L^2(\Omg^\ext)}\\
		&\qquad-
		\int_{\p\Omg} 
		\left[\p_\nu(\p_1 u)\ov{\p_1 v}\!+\!
		\p_\nu(\p_2 u)\ov{\p_2 v}\right]\dd\s
		\!+\!
		\int_{\p\Omg} 
		\left[\p_\nu(\Delta u)-\tau\p_\nu u+\gamma u\right]\ov{v}\dd\s,
		\end{aligned}
	\end{equation} 
	where the normal derivatives are computed with the unit normal pointing outwards of $\Omg$.
	In the coordinates $(s,t)$ we find
	for $j\in\{1,2\}$ using the expression~\eqref{eq:gradst} that
	\[
	\begin{aligned}
		\big(\p_\nu(\p_j u)\big)(s,0)& = \left(\p_t\left(
		\frac{\tau_j(s)\p_s u}{1+\kp(s)t} + \nu_j(s) \p_t u\right)\right)(s,0)\\
		&=
		\tau_j(s)\big(\p_{st} u\big)(s,0) - 
		\kp(s)\tau_j(s)\big(\p_s u\big)(s,0) +\nu_j(s)\big(\p_{tt}u\big)(s,0).
	\end{aligned}
	\]
	Moreover, we find for $j\in\{1,2\}$ that
	\[
		(\p_j v)(s,0) = \tau_j(s)\big(\p_s v\big)(s,0) + \nu_j(s)\big(\p_t v\big)(s,0).
	\]
	Hence, we obtain that
	\[
	\begin{aligned}
	&\big[\p_\nu(\p_1 u)\ov{\p_1 v}+
	\p_\nu(\p_2 u)\ov{\p_2 v}\big](s,0) =\\
	&\quad\!=\!
\left(\tau(s)\!
\left[\big(\p_{st}u
\big)(s,0)\!-\!\kp(s)\big(\p_s u\big)(s,0)\right]\! +\! \nu(s)\big(\p_{tt}u\big)(s,0)\right)
\!\cdot\!
\left(\tau(s)
\big(\ov{\p_s v}\big)(s,0)\! +\! \nu(s)\big(\ov{\p_t v}\big)(s,0)\right)\\
&\quad\!=\!\left[\big(\p_{st}u
\big)(s,0)-\kp(s)\big(\p_s u\big)(s,0)\right]\big(\ov{\p_s v}\big)(s,0)+\big(\p_{tt}u\big)(s,0)\big(\ov{\p_t v}\big)(s,0).
	\end{aligned}
	\]
	Using the above formulae and performing an integration by parts we can express the boundary term in~\eqref{eq:huv} as
	\[
	\begin{aligned}
		\frb[u,v] &:=-
		\int_{\p\Omg} 
		\left[\p_\nu(\p_1 u)\ov{\p_1 v}+
		\p_\nu(\p_2 u)\ov{\p_2 v}\right]\dd\s
		+
		\int_{\p\Omg} 
		\left[\p_\nu(\Delta u)-\tau\p_\nu u+\gamma u\right]\ov{v}\dd\s\\
		&=
		-\int_0^L
		\left[\big(\p_{st}u
		\big)(s,0)-\kp(s)\big(\p_s u\big)(s,0)\right]\big(\ov{\p_s v}\big)(s,0)\dd s
		-\int_0^L
		\big(\p_{tt}u\big)(s,0)\big(\ov{\p_tv}\big)(s,0)\dd s\\
		&\qquad\qquad +\int_0^L 
		\left[
		\big(\p_t(\Delta u)\big)(s,0)
		-\tau\big(\p_t u\big)(s,0)
		+\gg u(s,0)\right] \ov{v(s,0)}\dd s\\
		&=
		-\int_0^L
		\big(\p_{tt}u\big)(s,0)\big(\ov{\p_tv}\big)(s,0)\dd s\\
		&\qquad +\int_0^L 
		\left[
		\big(\p_s\left[\p_{st}u
		-\kp(s)\p_s u\right]-\tau\p_t u+
		\p_t(\Delta u)
		+\gg u\right](s,0) \ov{v(s,0)}\dd s.\\
		\end{aligned}
	\]
	By the first representation theorem we obtain that $u\in C^\infty_0(\ov{\Omg^\ext})$ belongs to $\dom\sfH_{\tau,\gg}^{\Omg^\ext}$ if, and only if, $\frb[u,v] = 0$ for any $v\in H^2(\Omg^\ext)$. Since by ~\cite[Thm. 1.5.2.1]{Gr} the range of the mapping $H^2(\Omg^\ext)\ni v\mapsto\{v(s,0),(\p_t v)(s,0)\}$ is dense in $L^2((0,L))\times L^2((0,L))$ we obtain that
	the condition
   $\frb[u,v] = 0$ for any $v\in H^2(\Omg^\ext)$ is equivalent to the fact that $u$ satisfies the boundary conditions in the formulation of the proposition.    
\end{proof}

\section{Auxiliary computations in the exterior of a disk}
\label{app:aux}
In this appendix we perform auxiliary computation in the exterior $\cB^\ext$ of the disk $\cB\subset\dR^2$ of radius $R >0$ centred at the origin. Throughout this appendix we work in the setting of Section~\ref{sec:ext_disk}.

Let $n,m\in\dZ$ be fixed.
Let $f,g\in L^2((R,\infty);r\dd r)$ be such that the functions 
\[
	u(r,\tt) = \frac{e^{\ii n\tt}}{\sqrt{2\pi}}f(r),\qquad
	v(r,\tt) = \frac{e^{\ii m\tt}}{\sqrt{2\pi}}g(r)
\]
belong in addition to the Sobolev space $H^2(\cB^\ext)$.
Our first aim is to
show orthogonality identities for these functions for $n\ne m$ and the second aim is to
 find a convenient expression for $\frh_{\tau,\gg}^{\cB^\ext}[u,v]$ for $n = m$. 

Using the expression for the gradient in polar coordinates given in~\eqref{eq:grad} we obtain
\begin{equation}\label{eq:polar}
\begin{aligned}
\int_{\cB^\ext}
\nabla \p_j u\ov{\nabla \p_jv}\dd x   
&=
\int_{\cB^\ext}
\left(\p_r \p_j u\ov{\p_r\p_j v} +
\frac{1}{r^2} \p_\tt\p_j u\ov{\p_\tt\p_j v}\right)\dd x,\qquad j\in\{1,2\},\\
\int_{\cB^\ext}
\nabla  u\ov{\nabla v}\dd x   
&=
\int_{\cB^\ext}
\left(\p_r  u\ov{\p_r v} +
\frac{1}{r^2} \p_\tt u\ov{\p_\tt v}\right)\dd x.
\end{aligned}
\end{equation}
Substituting the expression in polar coordinates for $\p_1$ given in~\eqref{eq:p1p2} we find the representation of $\p_r \p_1 u\ov{\p_r\p_1 v}$ in terms of $f$ and $g$
\[
\begin{aligned}
\p_r \p_1 u\ov{\p_r\p_1 v}
&=
\left(\cos\tt\p_{rr} u - \frac{\sin\tt\p_{r\tt}u}{r}
+ \frac{\sin\tt\p_\tt u}{r^2} \right)
\ov{\left(\cos\tt\p_{rr} v - \frac{\sin\tt\p_{r\tt}v}{r}
	+ \frac{\sin\tt\p_\tt v}{r^2} \right)}\\
& =
\frac{e^{\ii (n-m)\tt}}{2\pi}
\left(\cos\tt f'' - 
\ii n\sin\tt
\left(\frac{f'}{r}
- \frac{f}{r^2}\right) \right)
\ov{\left(\cos\tt g'' 
	-\ii m\sin\tt
	\left( \frac{g'}{r}
	- \frac{g }{r^2}\right) \right)}\\
&=\frac{e^{\ii(n-m)\tt}}{2\pi}
\bigg(\cos^2\tt f''\ov{g''}
+nm\sin^2\tt\left(\frac{f'}{r}-\frac{f}{r^2}\right)
\left(\frac{\ov{g'}}{r}-\frac{\ov{g}}{r^2}\right)\\
&\qquad +\sin\tt\cos\tt\left(-\ii n
\left(\frac{f'}{r}-\frac{f}{r^2}\right)\ov{g''} + \ii m f''\left(\frac{\ov{g'}}{r} -\frac{\ov{g}}{r^2}\right)\right)
\bigg). 	
\end{aligned}
\]
Analogously using the expression in polar coordinates for $\p_2$ given in~\eqref{eq:p1p2} we find the representation for $\p_r\p_2u\ov{\p_r\p_2v}$ in terms of $f$ and $g$
\[
\begin{aligned}
\p_r \p_2 u\ov{\p_r\p_2 v}
&=
\left(\sin\tt\p_{rr} u + \frac{\cos\tt\p_{r\tt}u}{r}
- \frac{\cos\tt\p_\tt u}{r^2} \right)
\ov{\left(\sin\tt\p_{rr} v + \frac{\cos\tt\p_{r\tt}v}{r}
	- \frac{\cos\tt\p_\tt v}{r^2} \right)}\\
& =
\frac{e^{\ii (n-m)\tt}}{2\pi}
\left(\sin\tt f'' + 
\ii n\cos\tt
\left(\frac{f'}{r}
- \frac{f}{r^2}\right) \right)
\ov{\left(\sin\tt g'' 
	+\ii m\cos\tt
	\left( \frac{g'}{r}
	- \frac{g }{r^2}\right) \right)}\\
&= \frac{e^{\ii(n-m)\tt}}{2\pi}
\bigg(\sin^2\tt f''\ov{g''} +nm\cos^2\tt
\left(\frac{f'}{r}-\frac{f}{r^2}\right)
\left(\frac{\ov{g'}}{r}-\frac{\ov{g}}{r^2}\right) \\
&\qquad +\sin\tt\cos\tt\left(\ii n\left(\frac{f'}{r}-\frac{f}{r^2}\right)\ov{g''} -\ii m f''\left(\frac{\ov{g'}}{r}-\frac{\ov{g}}{r^2}\right)\right)\bigg).
\end{aligned}
\]
Combining the above two expressions we get
\begin{equation}\label{eq:aux1}
	\p_r\p_1u\ov{\p_r\p_1v}
	+	\p_r\p_2u\ov{\p_r\p_2v} = 
	\frac{e^{\ii(n-m)\tt}}{2\pi}\left(f''\ov{g''} + nm \left(\frac{f'}{r}-\frac{f}{r^2}\right)\left(\frac{\ov{g'}}{r}-\frac{\ov{g}}{r^2}\right)\right).
\end{equation}
Next we find the expression for $\p_\tt\p_1 u\ov{\p_\tt\p_1 v}$ 
\[
\begin{aligned}
	\p_\tt\p_1 u\ov{\p_\tt\p_1 v} & =
	\left(\cos\tt\p_{r\tt}u-\sin\tt\p_r u
	-\frac{\cos\tt\p_\tt u +\sin\tt\p_{\tt\tt}u}{r}\right)
	\\
	&\qquad\cdot	\ov{	\left(\cos\tt\p_{r\tt}v-\sin\tt\p_r v
		-\frac{\cos\tt\p_\tt v +\sin\tt\p_{\tt\tt}v}{r}\right)}.\\
\end{aligned}
\]
Substituting the expressions for $u$ and $v$ in terms of $f$ and $g$ we get
\[
\begin{aligned}	
	\p_\tt\p_1 u\ov{\p_\tt\p_1 v}&=\frac{e^{\ii(n-m)\tt}}{2\pi}
	\left(
	\ii n\cos\tt f' -\sin\tt f' -\frac{(\ii n\cos\tt -n^2\sin\tt)f}{r}\right)\\
	&\qquad\cdot 
	\ov{\left(
	\ii m\cos\tt g' -\sin\tt g' -\frac{(\ii m\cos\tt -m^2\sin\tt)g}{r}\right)}\\
	&=\frac{e^{\ii(n-m)\tt}}{2\pi}
	\left(
	\ii n\cos\tt
	\left(f' -\frac{f}{r}\right) -\sin\tt
	\left( f'-\frac{n^2 f}{r}\right) \right)\\
	&\qquad\cdot 
	\ov{\left(
		\ii m\cos\tt
		\left(g' -\frac{g}{r}\right) -\sin\tt
		\left( g'-\frac{m^2 g}{r}\right) \right)}\\
	&= \frac{e^{\ii(n-m)\tt}}{2\pi}
		\bigg(mn\cos^2\tt\left(f'-\frac{f}{r}\right)
		\left(\ov{g'}-\frac{\ov{g}}{r}\right) +\sin^2\tt\left(f'-\frac{n^2f}{r}\right)\left(\ov{g'}-\frac{m^2\ov{g}}{r}\right)\\
		&\qquad +\sin\tt\cos\tt\left(-\ii n
		\left(f'-\frac{f}{r}\right)\left(\ov{g'}-\frac{m^2\ov{g}}{r}\right)+\ii m\left(f'-\frac{n^2f}{r}\right)	\left(\ov{g'}-\frac{\ov{g}}{r}\right)\right)\bigg).
\end{aligned}
\]
Analogously we obtain
\[
\begin{aligned}
	\p_\tt\p_2 u\ov{\p_\tt\p_2v} &=
	\left(\cos\tt\p_r u + \sin\tt\p_{r\tt}u + \frac{\cos\tt\p_{\tt\tt}u-\sin\tt\p_\tt u}{r}  \right)\\
	&\qquad \cdot\ov{	\left(\cos\tt\p_r v + \sin\tt\p_{r\tt}v + \frac{\cos\tt\p_{\tt\tt}v-\sin\tt\p_\tt v}{r}  \right)}.
\end{aligned}
\]
Substituting the expressions for $u$ and $v$ in terms of $f$ and $g$ we find
\[
\begin{aligned}
	\p_\tt\p_2u\ov{\p_\tt\p_2 v} &=
	\frac{e^{\ii(n-m)\tt}}{2\pi}
	\left(\cos\tt f' + \ii n\sin\tt f'
	+
	\frac{-n^2\cos\tt f-\ii n\sin\tt f}{r}
	\right)\\
	&\qquad \cdot
	\ov{\left(\cos\tt g' + \ii m\sin\tt g'
		+
		\frac{-m^2\cos\tt g-\ii m\sin\tt g}{r}\right)} \\
	&= \frac{e^{\ii(n-m)\tt}}{2\pi}
	\left(\ii n\sin\tt\left(f'-\frac{f}{r}\right)
	+
	\cos\tt\left(f'-\frac{n^2f}{r}   \right)\right)\\
	&\qquad\cdot\ov{	\left(\ii m\sin\tt\left(g'-\frac{g}{r}\right)
		+
		\cos\tt\left(g'-\frac{m^2g}{r}   \right)\right)}\\
	&=\frac{e^{\ii(n-m)\tt}}{2\pi}
	\bigg(mn\sin^2\tt\left(f'-\frac{f}{r}\right)\left(\ov{g'}-\frac{\ov{g}}{r}\right) + \cos^2\tt\left(f'-\frac{n^2f}{r}\right)\left(\ov{g'}-\frac{m^2\ov{g}}{r}\right)\\
	&\qquad +\sin\tt\cos\tt\left(\ii n
	\left(f'-\frac{f}{r}\right)\left(\ov{g'}-\frac{m^2\ov{g}}{r}\right)-\ii m\left(f'-\frac{n^2f}{r}\right)\left(\ov{g'}-\frac{\ov{g}}{r}\right)\right)\bigg).
\end{aligned}
\]
Combining the expressions for $\p_\tt\p_1u\ov{\p_\tt\p_1v}$ and
$\p_\tt\p_2u\ov{\p_\tt\p_2v}$
we obtain that
\begin{equation}\label{eq:aux2}
	\begin{aligned}
	&\frac{1}{r^2}\Big(\p_\tt\p_1u\ov{\p_\tt\p_1v}+\p_\tt\p_2u\ov{\p_\tt\p_2v}\Big) =\\
	&\qquad=	 \frac{e^{\ii(n-m)\tt}}{2\pi}\left(mn
	\left(\frac{f'}{r}-\frac{f}{r^2}\right)	\left(\frac{\ov{g'}}{r}-\frac{\ov{g}}{r^2}\right)+\left(\frac{f'}{r}-\frac{n^2f}{r^2}\right)\left(\frac{\ov{g'}}{r}-\frac{m^2\ov{g}}{r^2}\right)\right).
	\end{aligned}
\end{equation}
Finally, we find that
\begin{equation}\label{eq:aux3}
	\p_r u\ov{\p_r v} +\frac{1}{r^2}\p_\tt u\ov{\p_\tt v} = 
	\frac{e^{\ii(n-m)\tt}}{2\pi}\left(
	f'\ov{g'} + \frac{nm f\ov{g}}{r^2}
	\right).
\end{equation}
Combining~\eqref{eq:polar} with~\eqref{eq:aux1},~\eqref{eq:aux2} and~\eqref{eq:aux3} we infer that for $n\ne m$
\[
	\frh_{\tau,\gg}^{\cB^\ext}[u,v] = 0\qquad \int_{\cB^\ext}\left(\nb\p_1 u\ov{\nb\p_1v} + \nb\p_2 u\ov{\nb\p_2v}\right)\dd x =0,\qquad \int_{\cB^\ext}\nb u\ov{\nb v}\dd x =0,
\]
and moreover for $n = m$
\[	
\begin{aligned}
	\frh_{\tau,\gg}^{\cB^\ext}[u,v] & = 
	\int_R^\infty\bigg[f''\ov{g''} +
	\tau f'\ov{g'}+
	2n^2\left(\frac{f'}{r}-\frac{f}{r^2}\right)\left(\frac{\ov{g'}}{r}-\frac{\ov{g}}{r^2}\right)\\
	&\qquad\qquad\qquad + \left(\frac{f'}{r}-\frac{n^2f}{r^2}\right)\left(\frac{\ov{g'}}{r}-\frac{n^2\ov{g}}{r^2}\right)
	+\frac{\tau n^2f\ov{g}}{r^2}\bigg]r\dd r + \gg R f(R)\ov{g(R)}.
\end{aligned}	
\]
\end{appendix}

\newcommand{\etalchar}[1]{$^{#1}$}

\end{document}